\documentclass[reqno]{amsart}
\usepackage{xifthen,xparse,ifmtarg,suffix,xstring,iftex}
\usepackage{stackrel}
\usepackage{amsmath,amsthm,amssymb}
\usepackage[disallowspaces,fixamsmath]{mathtools}
\usepackage{parskip}
\usepackage{enumitem}
\usepackage{etoolbox}
\usepackage{xr}

\ifpdftex
    \usepackage{mathrsfs} 
    \DeclareMathAlphabet{\mathbrush}{T1}{pbsi}{xl}{n}
\else 
    \usepackage{unicode-math}
    \usepackage{fontspec}
    \ifluatex
        \DeclareMathAlphabet{\mathbrush}{T1}{pbsi}{xl}{n}
    \else
        \newcommand*{\mathbrush}[1]{\text{{\fontfamily{qzc}\selectfont #1}}}
    \fi
\fi

\usepackage[style=numeric,
            backend=biber,
            doi=true,
            sorting=nyt,
            clearlang=true,
            hyperref=true,
            arxiv=abs
            ]{biblatex}
\DeclareFieldFormat{url}{\href{#1}{\mkbibacro{URL}}} 
\renewbibmacro*{doi+eprint+url}{
  \printfield{doi}%
  \newunit\newblock%
  \iftoggle{bbx:eprint}
    {\usebibmacro{eprint}}
    {}%
  \newunit\newblock%
  \iffieldundef{doi}
    {\usebibmacro{url+urldate}}%
    {}}
\usepackage[unicode,psdextra]{hyperref}
\usepackage{hyperxmp}
\hypersetup{pdfauthor={Peter M. Gerdes},
            pdfsubject={Computability Theory, Turing Reducibility and r.e. sets},
            pdfcontactemail={gerdes@invariant.org},
            pdfcontacturl={http://invariant.org},
            pdflang={en},
            pdfcopyright={Copyright (C) \today, Peter M. Gerdes},
            baseurl={http://invariant.org},
            unicode=true,
            pdfdisplaydoctitle=true,
            pdfencoding=unicode,
            psdextra=true,
            pdfcontactaddress={Department of Mathematics,
                                \xmpquote{Indiana University\xmpcomma Bloomington\xmpcomma IN USA}}
}

\makeatletter
\usepackage[nameinlink]{cleveref}
\usepackage[steps]{rec-thy}

\theoremstyle{plain}
    \newtheorem{theorem}{Theorem}
    \newtheorem{corollary}[theorem]{Corollary}
    \newtheorem{lemma}{Lemma}[section]
\theoremstyle{definition}
    \newtheorem{definition}[lemma]{Definition}
    \newtheorem{question}{Question}
\theoremstyle{remark}


\DeclareMathOperator{\id}{id}

\let\hat=\widehat
\begin{document}

\title{The Tree Pulldown Method: McLaughlin's Conjecture and Beyond}
\author{Leo A. Harrington}
\address{University of California, Berkeley}
\email{leoanthony@berkeley.edu}

\author{Peter M. Gerdes}
\address{Indiana University, Bloomington}
\email{gerdes@invariant.org}


\date{}
\begin{abstract}
This paper finally fully elaborates the tree pulldown method used by one of us (Harrington) to settle McLaughlin's conjecture.  This method enables the construction of a computable tree \( T_0 \) whose paths are incomparable over \( \zeron{\alpha} \) and resemble \( \alpha \)-generics while leaving us almost completely free to specify the homeomorphism class of \( [T_0] \).  While a version of this method for  \( \alpha = \omega \) previously appeared in \cite{Simpson2016Implicit} we give the general construction for an arbitrary ordinal notation \( \alpha \).  We also demonstrate this method can be applied to a `non-standard' ordinal notation to establish the existence of a computable tree whose paths are hyperarithmetically incomparable and resemble \( \alpha \)-generics for all \( \alpha < \wck \).   Finally, we verify a number of corollaries including solutions to problems  57\( ^{*} \) , 62, 63 (McLaughlin's conjecture), 65 and 71 from \cite{Friedman1975One}.  
\end{abstract}

\maketitle
\section{History}

The results presented here were all first proved by one of the authors (Harrington), but were never published.  Despite solving a number of well-known open problems \textemdash specifically, problems 57\( ^{*} \) , 62, 63, 64, 65 and 71 from Friedman's famous list of open problems in mathematical logic \cite{Friedman1975One} \textemdash the primary reference for the tree pulldown method described in this paper has remained a set of mimeographed (and later scanned) notes that very briefly sketch the approach.  In this paper, we finally provide a detailed presentation of the general method (both for standard and non-standard ordinal notations) and prove (perhaps not always in the way originally imagined) all the claimed applications excepting 64 which requires a non-trivial additional construction beyond the scope of this paper. 

 Even though the full method has never been published, a version of this method for finite ordinals appeared in \cite{Simpson2016Implicit}  \textemdash providing a proof of problem 63 from \cite{Friedman1975One}.  Several unpublished manuscripts and mathoverflow posts have attempted to explain aspects of the construction or its applications.  We especially want to call attention to Hjorth's  detailed (but also unpublished) manuscript \cite{hjorth_argument_2003} which used the method in this paper to solve problem 57\( ^{*} \).  Hjorth chose to present the argument as a ``static set theoretical construction'' instead of the more dynamic computability theoretic construction given here and there are several interesting aspects of his construction that differ from those given in this paper.   Unfortunately, not all manuscripts are so high quality.  One of the authors (Gerdes) apologizes for the errors in their earlier manuscript presenting this approach \cite{gerdes_harringtons_2010}.  

It is worth putting the method presented here in context.  In \cite{Steel1978Forcing}, Steel introduced a notion of forcing \textemdash forcing with tagged trees \textemdash which builds a tree \( T \) along with a sequence of paths \( g_i \in [T] \) to solve many of these problems relative to a generic real (the tree \( T \)).  The tree-pulldown approach presented here can be thought of as an effective variant of this approach which allows for proofs of unrelativized variants of many of the theorems proved by forcing with tagged trees.  While the analogy is illuminating, it won't play any role in this paper.

This method is also closely connected to the study of \( \REA[\omega] \) sets.  For instance, a similar pulldown approach was used by one of the authors (Harrington) to construct an arithmetically low  \( \REA[\omega] \) set \( A \)  and a pair of arithmetically incomparable  \( \REA[\omega] \) which was later extended by Simpson to show that for any \( \REA[\omega] \) operator \( E(X) \) we can (effectively) find an \( \REA[\omega] \) set \( A \) such that \( E(A) \Tequiv \zeron{\omega} \) \textemdash a result that remains at the heart of our understanding of the arithmetic degrees of  \( \REA[\omega] \) sets.  It also remains important in the study of \( \pizn{2} \) singletons\footnote{Here we mean \( \pizn{2} \) singletons in \( \cantor \) (sets) which are equivalent \textemdash up to Turing degree \textemdash to \( \pizn{1} \) singletons in \( \baire \).} and was recently modified by one of the authors (Gerdes) to demonstrate the existence of a \( \pizn{2} \) singleton of minimal arithmetic degree \cite{Gerdes2023PiSingleton}.

\section{Notation}
We largely adopt standard notation as found in \cite{Odifreddi1999Classical} which we briefly review.  Set difference is denoted by \( X \setminus Y \),   the \( e \)-th set r.e. in \( X \) is  \( \REset(X){e} \) and the \( e \)-th computable functional applied to \( X \) by \( \recfnl{e}{X}{} \) with the first \( s \) steps of that computation denoted by \( \recfnl[s]{i}{X}{y} \).    We denote convergence and divergence by \( \recfnl{i}{X}{y}\conv   \) and \( \recfnl{i}{X}{y}\diverge   \) respectively. 

We write elements of \( \bstrs \) and \( \wstrs \) (referred to as strings) like \( \str{x_0, x_1, \ldots, x_{n-1}} \) with \( \estr \) denoting the empty string.  We call elements of \( \baire \) or \( \cantor \) reals and identify sets with their characteristic functions in \( \cantor \).  For strings and reals we denote (non-strict) extension by \( \subfun \), incompatibility by \( \incompat \), compatibility by \( \compat \), the concatenation of \( \sigma, \tau \)  by \( \sigma\concat\tau \)  and use \( \leftofeq \) to denote the lexicographic ordering (over \( <\)) while \( \leftof \) denotes the ordering left-of (\( \sigma \leftof \tau \iff \sigma \leftofeq \tau \land \sigma \incompat \tau \)).  The length of \( \sigma \) is written \( \lh{\sigma} \), \( \sigma^{-} \) refers to the immediate predecessor of \( \sigma \) under \( \subfun \) and we denote the string consisting of \( k \) copies of \( m \) by \( \str{m}^{k} \).  The longest common initial segment of \( \sigma \) and \( \tau \) is \( \sigma \meet \tau \).  We say \( f \) (string or real) meets a set of strings \( X \) \textemdash denoted \( f \supfun X \) \textemdash if \( f \supfun \tau \) for some \( \tau \in X \) and say \( f  \) strongly avoids \( X \) on \( T \) if there is some \( \sigma \subfun f \) such that no \( \sigma' \supfun \sigma \) is in \( T \isect X \).

We generally elide the distinction between strings, pairs and their codes but when necessary we write \( \godelnum{\pair{x}{y}} \), to indicate we mean the code.  Our coding will satisfy \( \godelnum{\estr} = 0 \) and \( \sigma \subfun \tau  \implies \godelnum{\sigma} \leq \godelnum{\tau} \).  We define \( A \Tplus B \), \( \setcol{X}{n} \) and \( \setcol{X}{< n}  \) standardly and let \( \Tplus_{n \in S} X_n = \set{\pair{n}{y}}{n \in S \land y \in X_n}  \) and extend this notion to functions in the obvious way.   We let \( \recfnl{e}{\sigma}{} \) denote the longest string  \( \tau  \) with \( \lh{\tau} \leq \lh{\sigma} \) such that \( \tau(n) = \recfnl[\lh{\sigma}]{e}{\sigma}{n}\conv \).  Trees will be downward closed subsets of \( \wstrs \) and define \( [T] \) to be the set of  paths which extend infinitely many elements of \( T \) (this takes on the usual meaning when \( T \) is a tree).   We write \( \pruneTree{T} \) for the set of strings in \( T \) extended by a path in \( [T] \) and abbreviate \( T \isect \set{\sigma \in \wstrs}{\lh{\sigma} \leq n} \) by \( T\restr{n} \).  

  \( \kleeneO \) denotes Kleene's set of ordinal notations ordered by  \( \Oless \) with \( \kleeneO+, \kleeneO- \) denoting the sets of successor and limit notations respectively.  We write \( \kleeneO1 \) for some  \( \piin{1} \) path \textit{through} \( \kleeneO \) \textemdash where  not to be confused with a path \textit{in}  \( \kleeneO \) which is any linearly ordered downward closed subset of \( \kleeneO \).  An apparent notation is a member of \( \kleeneO* \supset \kleeneO \) where \( \kleeneO* \) drops the well-foundedness requirement (see appendix \ref{app:defo}).  We write \( \Oadd \) for effective addition of notations, assume that \( 0 \) is the unique notation for the least ordinal and when \( \lambda \in \kleeneO- \) we write \( \Olim{\lambda}{n} \) for the \( n \)-th element in the effective limit defining \( \lambda \).  We denote the height of \( x \) under the partial order \( <_R \) by \( \hgt[<_R]{x} \) with the ordinal corresponding to the notation \( \alpha \) denoted \( \Ohgt{\alpha} \) and adopt standard open/closed interval notation, e.g., \( x \in [y, z) \) means \( y \leq_R x <_R z \) with \( R \) inferred from context. Finally, we define \( \jumpn{X}{\alpha}, \alpha \in \kleeneO \) to satisfy \( \TPlus_{\gamma \Oleq \beta} \setcol{\jumpn{X}{\alpha}}{\gamma} = \jumpn{X}{\beta} \) whenever \( \beta \Oleq \alpha \) and assume\footnote{Note that we can help ourselves to \( \zeroj \) to distinguish between non-zero notations.} that we can computably recover \( \beta \) from \( \jumpn{X}{\beta} \).

  \section{Preliminaries}  

  Before we can state the main results of this paper, we will need a few definitions.

  \begin{definition}\label{def:alpha-equivalent}
  \( X \) is \( \beta \) reducible to \( Y \)   \textemdash denoted \( X \Tleq^{\beta} Y \) \textemdash if \( \jumpn{Y}{\beta} \Tgeq X \) with \( \beta \) equivalence, comparability and incomparability defined in the natural way.
\end{definition}
    
    Some care is warranted here as \( \beta \) reducibility (and hence equivalence) is only transitive when \( \beta \) is an additively indecomposable limit ordinal but this won't matter for our purposes.

  \begin{definition}\label{def:monotone}
    A function \( \Gamma \) from strings to strings is monotone if whenever \( \Gamma(\tau)\conv \) we have \( \sigma \subfunneq \tau \) iff \( \Gamma(\sigma) \subfunneq \Gamma(\tau) \).  A monotone function that respects \( \leftof \) is an order-preserving function and if an order preserving function respects \( \meet \) (longest common initial segment) it is an expansionary function.     
  \end{definition}

  Observe that the definition of monotone requires that the domain of \( \Gamma \) be downward closed and that an expansionary function is one where \( \Gamma(\sigma\concat[x]) \supfun \Gamma(\sigma)\concat[x'] \) (and is order preserving) thereby justifying the name.  When dealing with a function \( \Gamma \)  from strings to strings we abbreviate \( \Gamma\restr{\set{\sigma}{\lh{\sigma} \leq n}} \) by \( \Gamma\restr{n} \) just as we do with trees.

\begin{lemma}\label{lem:expan-homeo}
If \( \Theta \) is a monotone function on strings then \( \Theta  \) induces a functional with domain \( [\dom \Theta] \) which is homeomorphically mapped to  \( [\rng \Theta] \).  If \( \Theta \) is expansionary and \( T \) is the downward closure of \( \rng \Theta \) then \( \Theta \) is a homeomorphism of \( [\dom \Theta] \) and  \( [T]  \).        
\end{lemma}
\begin{proof}
   Monotonicity guarantees that if \( l_{i+1} > l_i \) then \( \Theta(g\restr{l_0}) \subfunneq \Theta(g\restr{l_1})\subfunneq \ldots \).  Thus \( \Theta(g)\conv \)  iff \( \existsinf(l)\Theta(g\restr{l})\conv \) iff \( g \in [\dom \Theta] \) iff \( \Theta(g) \in [\rng \Theta] \).  Monotonicity ensures that every real in \( [\rng \Theta] \) is the image of a unique real in \( [\dom \Theta] \) and from this it is easy to verify \( \Theta \) is bicontinuous and therefore a homeomorphism of \( [\dom \Theta] \)  with  \( [\rng \Theta] \). 

    However, since \( \rng \Theta \) may fail to be downward closed \textemdash recall that \( [\rng \Theta] \) is the set of paths which extend infinitely many strings in \( \rng \Theta \) \textemdash this doesn't guarantee that \( \Theta \) is a homeomorphism of its domain with the downward closure of it's range \textemdash  the standard map \( \zeta \)   that sends \( \baire \) to  the coinfinite sets in \( \cantor  \), will be monotone while the downward closure of \( \rng \zeta \) will be the full binary tree.  However, if \( \Theta \) is expansionary then  \( [\rng \Theta]=[T] \).
\end{proof}

  In light of this lemma, we will elide the distinction between monotone functions on strings and the induced functionals as the notation we've adopted for \( \recfnl{e}{\sigma}{} \)  already interprets computable functionals as monotone functions on strings.

\begin{definition}\label{def:subgeneric}
 A degree \( \mathbf{g}  \)  is \( \alpha \)-subgeneric for \( \alpha < \wck \) if for all \( \beta' \Oless \beta \Oleq \alpha \) 
 \begin{flalign}
  &\jumpn{\mathbf{g}}{\beta} = \zeron{\beta} \Tjoin  \mathbf{g} \label{eq:subgeneric:jump} \\
  &\jumpn{\mathbf{g}}{\beta'} \Tmeet \zeron{\beta} = \zeron{\beta'} \label{eq:subgeneric:meet} \\
  &\mathbf{g} \nTleq \zeron{\beta'} \label{eq:subgeneric:non-trivial} 
  \end{flalign}
  \( \mathbf{g} \) is \( \wck \)-subgeneric, or simply subgeneric, if the above holds for all \( \alpha < \wck \). 
      
\end{definition}

We engage in the obvious abuse of notation and talk about a real \( f \)  being \( \alpha \)-subgeneric for an ordinal notation \( \alpha \) when the degree of \( f \) is \( \Ohgt{\alpha} \)-subgeneric.  It will be useful to talk about classes whose members uniformly satisfy the definition of \( \alpha \)-subgeneric and whose \( \alpha \) jumps are incomparable \textemdash we explicate this notion below.

\begin{definition}\label{def:uniformly-subgeneric}
A class \( \mathcal{C} \) of reals is uniformly  \( \alpha \)-subgeneric  for \( \alpha \in \kleeneO \)  if there is a computable functional \( \Gamma \) such that  
\begin{enumerate}
    \item\label{def:uniformly-subgeneric:subg}  Every \( g \in \mathcal{C} \) is \( \alpha \)-subgeneric 
    \item\label{def:uniformly-subgeneric:uniform}  For all \( g\in \mathcal{C} \) and \(  \beta \Oleq \alpha \) we have \( \Gamma(\zeron{\beta};g) = \jumpn{g}{\beta} \).
    \item\label{def:uniformly-subgeneric:incomp}  If \( f, g \in \mathcal{C} \), \( \beta \Oless \alpha \)  and \( f \neq g \) then \( f \nTleq \jumpn{g}{\beta}  \).   
\end{enumerate}      
The class is uniformly \( \wck \)-subgeneric \textemdash or just uniformly subgeneric \textemdash  if the above holds for all \( \alpha  \) in some path \( \kleeneO1 \) through \( \kleeneO \).    
\end{definition}

We will say a tree \( T \) is uniformly \( \alpha \)-subgeneric if \( [T] \) is a uniformly \( \alpha \)-subgeneric class.

\section{Main Theorems}

We can now state the two primary theorems.

\begin{theorem}[Main Theorem - Standard Version]\label{thm:standard}
For all \( \alpha \in \kleeneO \) and tree \( T \Tleq \zeron{\alpha}  \) there is a computable tree \( T_0 \subset \wstrs \) such that \( [T_0] \) is a uniformly \( \alpha \)-subgeneric class and there is a \( \zeron{\alpha} \) computable expansionary function \( \Gamma\maps{T}{T_0} \) that gives a homeomorphism of \( [T] \) and \( [T_0] \).  This holds with all possible uniformity \textemdash including uniformity in \( \alpha \) and \( T \).   
\end{theorem} 

We now extend this result to produce a uniformly subgeneric (i.e. \( \wck \)-subgeneric) class.  This extension can be viewed as the result of applying  \cref{thm:standard} to non-standard ordinal notations.  However, moving to non-standard ordinal notations also adds `phantom' paths in \( T_0 \) that aren't images of elements in \(  T \).  This consideration yields the following theorem.

\begin{theorem}[Main Theorem - Non-standard Version]\label{thm:non-standard}
There is a computable tree \( T_0 \subset \wstrs \) such that \( [T_0] \) is uniformly subgeneric and contains a \textit{subset} homeomorphic to \( \baire \) via an expansionary function \( \Gamma\maps{\wstrs}{T_0} \).  
\end{theorem} 

Obviously, this implies \( [T_0] \) contains a subset homeomorphic via \( \Gamma \)  to any subtree of \( \wstrs \) as well.

\section{Applications}

We now list some of the applications of these theorems.  As with the primary result, these are also all due to Harrington.  We start by providing a a solution to problem 57\( ^{*} \)  of \cite{Friedman1975One}.

\begin{corollary}[Solution to 57\( ^{*} \)]\label{cor:uncountable-not-every-hyperdegree}
There is a computable tree \( T \subset \wstrs \) with \( [T] \) uncountable such that for all \( f \in [T], \lambda \in \kleeneO \) we have \( \jumpn{f}{\lambda} \nTgeq \kleeneO \).  Thus, \( T \) fails to have any path of the same \( \lambda \) degree (when defined) as \( \kleeneO \).    
\end{corollary}
Here \( \lambda \) degrees refers to the degrees formed by the notion of \( \lambda \) reducibility defined above.  
\begin{proof}
Take \( T \subset \wstrs \) to satisfy \( [T] \) is uniformly subgeneric as guaranteed by \cref{thm:non-standard}.  If  \( f \in [T] \) satisfied \( \jumpn{f}{\lambda} \Tgeq  \kleeneO  \)   then \textemdash as \( \kleeneO   \Tgeq \zeron{\lambda \Oadd 1} \) \textemdash  \( \zeron{\lambda + 1} \Tleq \jumpn{f}{\lambda} \).   As \( f \) is \( \lambda + 1 \)-subgeneric, we would have \( \zeron{\lambda + 1} \Tleq \zeron{\lambda} \).  Contradiction.        
\end{proof}

Another corollary solves problem 62 in \cite{Friedman1975One}. 

\begin{corollary}[Solution to 62]\label{cor:H-set-non-arith}
There is a \( \pizn{1} \) function singleton \( g \) such that \( g \nAequiv \zeron{\beta} \) for any \( \beta < \wck \).      
\end{corollary}
Here \( \Aequiv \) denotes arithmetic equivalence, i.e., the equivalence relation induced by the relation of relative arithmetic definability. 
\begin{proof}
Applying \cref{thm:standard} with \( \alpha = \omega \) and \( T = \set{\str{0}^{k}}{k \in \omega} \)  gives us a \( \pizn{1} \) function singleton \( g \)  that is an \(  \omega \)-subgeneric.  Now suppose, for a contradiction, that \( g \Aequiv \zeron{\gamma} \) for some \( \gamma < \wck \) then we must have \( g \Tleq \jumpn{\zeron{\gamma}}{n}  \) for some \( n \in \omega \).  Thus, we must have \( \gamma + n \geq \omega \) by \cref{eq:subgeneric:non-trivial} and thus \( \gamma \geq \omega \).  Hence \( \zeron{\omega} \Tleq \zeron{\gamma}  \Tleq \jumpn{g}{n} \).  But by  \cref{eq:subgeneric:meet} any set arithmetic in \( g \) is actually arithmetic contradicting the fact that \( \zeron{\omega} \) isn't arithmetic.
\end{proof}

The next two corollaries have previously appeared in print in \cite{Simpson2016Implicit}.  

\begin{corollary}[Arithmetically Incomparable Singletons]
There is a pair of \( \pizn{1} \) arithmetically incomparable function singletons \( f, g \) with \( \jumpn{f}{\omega} = \jumpn{g}{\omega} = \zeron{\omega}  \).  Equivalently there is a pair of arithmetically incomparable and arithmetically low \( \pizn{2} \) set singletons . 
\end{corollary}
\begin{proof}
A straightforward argument shows that \( f \in \baire \) is the unique solution of some \( \pizn{1}  \) formula iff it's graph is the unique solution of some  \( \pizn{2}  \) formula \textemdash the extra quantifier is needed to assert that \( X \) codes a total function but see \cite{Gerdes2023PiSingleton} for formal proof (using a slightly different coding).  Thus, it is enough to demonstrate the main claim.

Let \( T \) be a computable tree without terminal nodes that has exactly two paths.  By \cref{thm:standard} we can find a computable tree \( T_0 \) such that \( [T_0] \) is homeomorphic to \( [T] \) and is uniformly \( \omega \)-subgeneric.   Let \( f, g \) such that \( [T_0] = \set{f,g} \).  As both \( f \) and \( g \) are isolated paths on \( T_0 \) they are both the unique path through some computable tree \textemdash hence  \( \pizn{1} \) function singletons.  The uniform subgenericity of \( [T_0] \) guarantees that \( f \) and \( g \) are arithmetically incomparable and since they are both the image of computable paths under a \( \zeron{\omega} \) computable functional we have that \( f, g \Tleq \zeron{\omega}  \).  As \( \jumpn{f}{\omega} \Tequiv \zeron{\omega} \Tplus f  \) \textemdash and likewise for \( g \) \textemdash we have \( \jumpn{f}{\omega} \Tequiv \jumpn{g}{\omega} \Tequiv \zeron{\omega} \).              
\end{proof}

The same approach can be adapted to construct arithmetically incomparable \( \REA[\omega] \) sets \textemdash a result first established by one of us (Harrington) but only distributed as unpublished notes \cite{Harrington1976Arithmetically}.  The first published version of the construction seems to be \cite{Simpson1985Arithmetic} but a more accessible version of the result can be found in \cite{Odifreddi1999Classical} (theorem XIII.3.5).  The next corollary gives a negative answer to McLaughlin's conjecture (problem 63 of \cite{Friedman1975One}).

\begin{corollary}[Solution to 63]\label{cor:pi-class-no-arith-singleton}
There is a countable \( \pizn{1} \) subset of \( \baire \) containing an element which isn't an arithmetic singleton.
\end{corollary}

Proving this result requires we establish a few lemmas first.  We start by showing that we can define something like a notion of forcing for uniformly subgeneric classes of reals.  

\begin{lemma}\label{lem:subgeneric-forcing}
If \(g \in  \mathcal{C} \) where \( \mathcal{C} \)  is a uniformly \( \alpha \)-subgeneric  class and  \( \psi \in \sigmazn{\beta} \union \pizn{\beta},  \beta \Oleq \alpha \) there is some \( \sigma \subfun g \) such that either all \( h \in \mathcal{C}, h \supfun \sigma \) satisfy \( \psi \) or no \( h \) does.  Moreover, such a \( \sigma \) can be found computably in \( \psi, \zeron{\beta} \Tplus g \).         
\end{lemma}

See \cite{Ash2000Computable} for a definition of \( \sigmazn{\beta} \) and \( \pizn{\beta} \) formula in terms of computably infinitary formulas.  Also note that the above result implies that if \( \mathcal{C} \) is uniformly subgeneric then all hyperarithmetic facts are `forced' for paths in  \( \mathcal{C} \) with the uniformity claim holding provided \( \beta \) is in the unique path through \( \kleeneO \) witnessing \cref{def:uniformly-subgeneric}.

\begin{proof}
Given \( \psi \in \sigmazn{\beta} \union \pizn{\beta}  \) we can decide if if \( g \models \psi \) computably in \( \jumpn{g}{\beta} \).  Using the computable functional \( \Theta \) from \cref{def:uniformly-subgeneric} we can construct a computable function \( \rho \) such that \( \rho(g \Tplus \zeron{\beta}) \) determines if \( g \models \psi \) for all \( g \in \mathcal{C} \).  Any string \( \sigma \subfun g \) longer than the use of \( \rho \) then witnesses the claim and the moreover claim follows from the computability of \( \rho \).  
\end{proof}

We can now prove the existence of a countable \( \pizn{1} \) class whose elements aren't all arithmetic singletons.

\begin{proof}
Let \( T \) be a computable tree such that \( [T] \) is countable and contains a non-isolated branch, e.g., take \( T = \set{\sigma}{\lnot \exists(x)\exists(y \neq x)\left(\sigma(x) = \sigma(y) = 1 \right)} \).  By \cref{thm:standard} there is a computable tree \( T_0 \) such that \( [T_0] \) is uniformly \( \omega \)-subgeneric and homeomorphic to \( [T] \).  Let \( g  \) be a non-isolated element in \( [T_0] \).   

Given an arithmetic formula \( \psi \) such that \( g \models \psi \) let \( \sigma \subfun g \) be the string witnessing \cref{lem:subgeneric-forcing}.  As \( g \) is non-isolated there is some \( f \in [T_0]  \) with \( f \supfun \sigma \).  Thus, \( f \models \psi \) demonstrating that \( g \) can't be the unique solution of any arithmetic formula.    
\end{proof}

So far all of the applications have used the standard version of the main theorem.  The final two applications require the non-standard version.

\begin{corollary}[Solution to 65]\label{cor:L-hyp-not-deltaii}
There is an \( \omega \)-model \( M \)  of \( \sigmain*{1}\)-DC  such that \( M \not\models X \in \deltain{1} \implies X \in \HYP  \)  
\end{corollary}

Here we are working in a 2-sorted structure in the language of second-order arithmetic with  number variables ranging over \( \omega \) and the set variables ranging over elements in \( \HYP[g] \).  When dealing with models of second-order arithmetic we identify functions with the set giving their graph.

\begin{proof}
We start by applying \cref{thm:non-standard} to produce a computable tree \( T \) such that \( [T] \)  is mutually subgeneric (and \( [T] \neq \eset \)).  As suggested in \cite{mathoverflow65} we can apply the Gandy Basis theorem   to find an element \( g \in [T] \) with \( \wck[g] = \wck \).  We now argue that \( \HYP[g] \) provides an \( \omega  \)-model of  \( \sigmain*{1} \)-DC.  

The argument that \( \HYP[g] \) models  \( \sigmain*{1} \)-DC is a straightforward relativization of the argument that \( \HYP \) models  \( \sigmain*{1} \)-DC.  Now we argue that \( g \) has a \( \deltain{1} \) definition in \( \HYP[g] \).  Using \( T(\sigma)  \) to represent the computable predicate for membership in \( T \) we claim that 
\begin{align*}
g(x) = k &\iff \forall(h \in  \HYP[g])\exists(n > x)\left(h(x) = k \lor \lnot T(h\restr{n})  \right) \\
 & \iff \exists(h \in \HYP[g])\forall(n)\left(T(h\restr{n}) \land h(x) = k \right)   
\end{align*}
This follows from the uniform subgenericity of all paths through \( T \) guaranteeing that \( g \) is the only real in \( \HYP[g] \) that is a path through \( [T] \). This becomes a \( \deltain{1} \) definition over \( \HYP[g] \) simply by omitting the restrictions.

Now suppose, for a contradiction, \( \HYP[g] \models g \in \HYP \).  That is, 
\[ \HYP[g]  \models \exists(\alpha \in \kleeneO)\exists(X)\left(H(X, \alpha) \land  g \Tleq X \right) \] 
where \( H(X, \alpha) \) is the \( \pizn{2} \) condition\footnote{See theorem 4.2 of chapter 2 of \cite{Sacks1990Higher} for a definition of \( H \).}  guaranteed to be uniquely satisfied by \( \zeron{\alpha} \) if \( \alpha \in \kleeneO \).  Let \(\alpha,  X \in \HYP(g) \) witness the truth of the supposition.

As \( g \nin \HYP \) any purported notation \( \alpha \) witnessing the above can't actually be in \( \kleeneO  \).  However, analyzing the definition of \( \kleeneO \) (see \cref{eq:def-o} in appendix \ref{app:defo}) reveals that the only way \( \HYP(g) \models \alpha \in \kleeneO \)  when it isn't is if there was some infinite \( \Oless \)  decreasing sequence below \( \alpha \) but no such sequence hyperarithmetic in \( g \).  But (see lemma 2.2 of chapter 3 of \cite{Sacks1990Higher}) this guarantees that \( \set{\beta \Oless \alpha} \) has a well-founded initial segment of length \( \wck \) and (see \cref{eq:def-o} again) this guarantees that \( \alpha \) bounds some unique path of notations through \( \kleeneO \).  As \( X \in \HYP(g) \) choose \( \beta \Oless \alpha \) to be such that \( X \Tleq \jumpn{g}{\beta} \).    

Using the inductive definition of \( H(X, \alpha) \) we know that since \( \beta \Oadd 1 \Oless \alpha \) and \( H(\zeron{\beta \Oadd 1}, \beta \Oadd 1) \) we must have \( \zeron{\beta \Oadd 1} \Tleq X \Tleq \jumpn{g}{\beta}  \).   However, this contradicts the subgenericity of \( g \) as  \( \zeron{\beta \Oadd 1} \nTleq \zeron{\beta} \).  Thus, \( \HYP[g] \) satisfies  \( \sigmain*{1}\)-DC  but doesn't satisfy that every \( \deltain{1} \) set is hyperarithmetic.

\end{proof}

Our final application is to solve problem 71 of \cite{Friedman1975One}.   In doing so we will take the opportunity to also state some results we will need later.

\subsection{Problem 71}\label{ssec:prob71}

\begin{corollary}[Solution to 71]\label{cor:path-through-O}
There is a path \( \kleeneO1 \) through \( \kleeneO \) such that every \( X \in \HYP  \) with \( X \Tleq \kleeneO1  \) is computable.   
\end{corollary}

Our approach will be to use the  Kleene-Brouwer ordering to linearise the elements of a tree with only subgeneric paths.  We will then argue that, as the tree has no hyperarithmetic paths, that we can identify the well-founded part of the tree to the left of all infinite paths with some path through \( \kleeneO \).  This will demonstrate that the leftmost path through our tree is a subgeneric degree and also the degree of a path through \( \kleeneO \).

This argument in this subsection draws heavily from the construction of a path through \( \kleeneO \) in  ch. 3 \S 2 of \cite{Sacks1990Higher}.  Unfortunately, we can't simply cite to this result as the construction there only guarantees the path is r.e. in the well-ordered initial segment.  We start by recalling the following definition.

\begin{definition}\label{def:kb-order}
The Kleene-Brouwer ordering on strings is defined by
\[ \sigma \KBless \tau \iff \sigma \supfunneq \tau \lor \sigma \leftof \tau \]
where \( \leftof \) means left-of 
\end{definition}

This is clearly a linear order on \( \wstrs \) and we now observe that it preserves well-foundedness facts in an effective fashion.

\begin{lemma}\label{lem:kb-wellfounded}
Suppose \( T \) is a tree then 
\begin{itemize}
\item The relation \( \KBless \) restricted to \( T \)  is a well-order iff \( [T] = \eset \) and admits an infinite descending hyperarithmetic sequence iff \( [T] \) has an infinite hyperarithmetic path.  

\item If \( f \) is the leftmost path through \( T \) then   \( T_{\leftof f} \eqdef \set{\sigma}{\sigma \in T \land \sigma \leftof f} \) is the maximal well-ordered initial segment of \( T \) under \( \KBless \).   

\item If \( T \in \HYP \) and \( [T] \neq \eset = [T] \isect \HYP \) then \( \KBless \) restricted to \( T_{\leftof f} \) has height \( \wck \).  
\end{itemize}   
\end{lemma}
\begin{proof}
The first claim is proved as lemma 2.1 in Ch. 3 in \cite{Sacks1990Higher}.  For the second claim we note that the cited proof shows that if \( \sigma_i \) is an infinite decreasing sequence (under \( \KBless \)) then there is an infinite compatible subsequence (necessarily of unbounded length).  Now if \( \sigma_i \) was an infinite decreasing sequence in \( T_{\leftof f} \) then there would be some \( g \in [T] \) with \( g \leftof f \) contradicting the fact that \( f \) was the leftmost path.  

Thus, \( T_{\leftof f} \) is well-ordered by \( \KBless \).  As any \( \sigma \in T \setminus T_{\leftof f} \) bounds the infinite decreasing sequence given by \( f \), \( T_{\leftof f} \) is the maximal well-ordered initial segment.   Finally, the moreover claim is proved in Ch 3. lemma 2.2  of \cite{Sacks1990Higher} for computable trees and this result follows immediately by relativization to \( T \) since the ordinals constructive in \( T \) are equal to the constructive ordinals.    
\end{proof} 

We note the following useful property of linear orderings without hyperarithmetic descending chains.  

\begin{lemma}\label{lem:hyp-set-has-min}
If \( < \in \HYP \) is a linear ordering on a computable domain \( D  \)  with no infinite hyperarithmetic descending sequence then every \( \piin{1} \) subset  of \(  D \) has a \( < \) least element.  Also, every \( \piin{1} \) subset of \( D^{n}, n \in \omega \) has a lexicographically (over \( < \)) minimal element.    
\end{lemma}
This result will allow us to pretend \( < \) is a well-order and argue via induction when we are dealing with hyperarithmetic properties \textemdash which we will often do without explicit reference to this result. 
\begin{proof}
The main claim is proved in \cite{harrison_thesis} but we will only use it in the case where \( X \in \HYP \) which is what we prove.  If the claim failed then we would have. 
\[ \forall(z)\exists(y)\left(z \nin X \lor \left[y \in X \land y < z\right]\right) \]  
But the relationship between \( z \) and \( y \) is clearly a \( \piin{1} \) property if \( X \in \HYP \)  so we can uniformize it with a hyperarithmetic function \( f \) satisfying
\[ \forall(z)\left(z \nin X \lor \left[f(z) \in X \land f(z) < z\right]\right) \]  
Starting with some \( z_0 \in X \) the sequence \( z_{i+1} = f(z_i) \) defines an infinite hyperarithmetic decreasing sequence contradicting our assumption.  The generalization to the lexicographic order defined using \( < \) is immediate.  
\end{proof}

We will soon show that if \( T \) is a computable tree whose paths are uniformly subgeneric and \( f \) is the leftmost path through \( T \)  then  we can computably map \( T_{\leftof f}  \) to the set of limit notations on a path through \( \kleeneO \).  This would be enough for our purposes here, but later we will want to have the image of \( T_{\leftof f} \) to be a full path through \( \kleeneO \).  To this end, we will add strings to \( T \) without adding paths so these new strings will fill in the missing successor notations.

\begin{definition}\label{def:shift-right}
Given a tree \( T \) define  \( T^{+}  \) to be \( T \union \set{\sigma\concat[n]}{\sigma \in T \land n \in \omega} \).  

Also define \( \sigma^{+} \) for \( \sigma \in T^{+} \) as follows so that \( \sigma^{+} \) will be the immediate successor of \( \sigma \) under \( \KBless \)  in \( T^{+} \).  
\begin{equation*}
    \sigma^{+} \eqdef \begin{cases}
                        \diverge & \text{if } \sigma \nin T^{+} \\
                        \eta(\estr) & \text{if } \sigma = \estr \\
                        \eta(\sigma^{-}\concat[x+1]) &  \text{o.w. where } x = \sigma(\lh{\sigma}-1)
                        \end{cases}
\end{equation*}
Where \( \eta(\sigma) \) is defined to be \( \sigma\concat[0^k] \) for least \( k \in \omega \) (possibly \( 0 \)) so that  \( \eta(\sigma) \nin T \).  
\end{definition}

Whenever \( \sigma \in T^{+} \) then we will have \( \eta(\sigma) \in T^{+} \setminus T \) and \( \eta(\sigma) \) will be the \( \KBless \)  least \( \tau \supfun \sigma \) in \( T^{+} \).   We now map \( T^{+} \) into a compatible set of (apparent) ordinal notations.

\begin{lemma}\label{lem:kb-to-O}
Suppose that \( T \) is a computable tree without hyperarithmetic branches and \( f \) is the leftmost path through \( T \) then there is a  computable function \( \rho \) such that 
\begin{enumerate}
 \item\label{lem:kb-to-O:dom}   \( \dom \rho = T^{+}  \) 
 \item\label{lem:kb-to-O:mono}  \( \rho\restr{T} \) is strictly monotonically increasing (under \( < \) applied to codes not \( \KBless \)) .
 \item\label{lem:kb-to-O:onto} If \( \sigma \in T^{+}_{\leftof f}  \) then \( \rho(\sigma) \in \kleeneO \) and \( \hgt{\rho(\sigma)} = \hgt{\sigma}_{\KBless} \).
 \item\label{lem:kb-to-O:w-restr} \( \sigma \in T  \) iff  \(  \rho(\sigma) \) is a limit notation.  
 \item\label{lem:kb-to-O:comp} For all \( \sigma, \tau \in T^{+} \),  \( \rho(\sigma)  \) and \( \rho(\tau) \) are comparable under \( \Oleq \) .     
\end{enumerate}

\end{lemma}
Note that parts \ref{lem:kb-to-O:onto} and \ref{lem:kb-to-O:comp} together ensure that \( \rho \) is order preserving on the well-ordered initial segment of \( T \).  Also note that parts\ref{lem:kb-to-O:onto} and \ref{lem:kb-to-O:w-restr}  guarantee that if \( \sigma \in T^{+}\setminus T \) then \( \rho(\sigma) \) is a (apparent) successor notation.  The reason for \ref{lem:kb-to-O:mono} is to ensure that we can compute the image of \( T^{+}_{\leftof f}   \) under \( \rho \) from  \( T^{+}_{\leftof f}  \) \textemdash otherwise the image risks merely being r.e. in the domain.    

\begin{proof}
By the recursion theorem we may assume that we have access to an index \( i \) for the functional \( \rho \) we are building.  We denote the functional given by that index as \( \hat{\rho} \) and define

\begin{equation}\label{eq:rho}
\rho(\sigma) \eqdef \begin{dcases}
                        \diverge & \text{if } \sigma \nin T^{+} \\
                       \lim_{n\to\infty} \hat{\rho}(\sigma\concat[n])  & \text{if } \sigma \in T \\
                        0 & \text{if } \sigma = \eta(\estr) \\
                        \hat{\rho}(\tau) \Oadd 1 & \text{otherwise where } \tau^{+} = \sigma \\
                    \end{dcases}
\end{equation}

It is relatively straightforward to see that if \( \sigma \in T^{+}\setminus T \) and \( \sigma \neq \eta(\estr) \) then we can find some \( \tau \) with \( \tau^{+} = \sigma \).  Moreover, by choosing an appropriately large index for the effective limit we can ensure that part \ref{lem:kb-to-O:mono} is satisfied.  Since we only use \( \hat{\rho} \) in the definition above applied to nodes that are smaller under \( \KBless \) a straightforward induction shows that \( \rho \) will be defined on all (and only) elements in \( T^{+} \) verifying part \ref{lem:kb-to-O:dom}.   A similar induction lets us vindicate \ref{lem:kb-to-O:onto} and, by construction,  \( \rho(\sigma) \) is always a limit notation when \( \sigma \in T \)  and a successor notation when \( \sigma \in T^{+} \setminus T \) verifying part \ref{lem:kb-to-O:w-restr}.

This leaves us to prove part \ref{lem:kb-to-O:comp}.  Suppose that the claim fails and \textemdash by \cref{lem:hyp-set-has-min} \textemdash let \( \sigma \KBless \tau \) be the  lexicographically least pair under \( \KBless \) which witnesses the failure.  WLOG we must have both \( \sigma, \tau \in T \) since if either was in \( T^{+} \) then considering their predecessor under \( \KBless \) would contradict the minimality.  However, in this case, \( \rho(\tau) \) is equal to the limit of \( \rho(\tau\concat[n]) \) and for some \( n \) we must have \( \rho(\sigma) \Oless \rho(\tau\concat[n]) \).  Therefore,  \( \rho(\sigma) \Oless \rho(\tau) \) contradicting our assumption.                  
\end{proof}  

We can now prove the corollary.  Recall that we must show there is a path \( \kleeneO1 \) through \( \kleeneO \) such that every \( X \in \HYP  \) with \( X \Tleq \kleeneO1  \) is computable.   

\begin{proof}
By \cref{thm:non-standard} let \( T \) be a computable tree such that \( [T] \neq \eset \) is uniformly subgeneric.  As \( T^{+} \) has only  non-hyperarithmetic paths by \cref{lem:kb-wellfounded} we have that \( T^{+}_{\leftof f} \) is the well-ordered initial segment of \( \KBless \) on \( T^{+} \) and has height \( \wck \).  Letting \( \rho \) be as in \cref{lem:kb-to-O} immediately entails that the image of \( T^{+}_{\leftof f}\) under \( \rho \) \textemdash call it \( \kleeneO1 \) \textemdash is path in \( \kleeneO \) with height \( \wck \) and, thus, a path through \( \kleeneO \).  

The monotonicity of \( \rho \) on \( T \)  guarantees that the image of \( T_{\leftof f} \) under \( \rho \) is Turing equivalent to \( T_{\leftof f} \Tequiv f \).  It is easy to verify that \( \kleeneO1 \) \textemdash the image of \( T^{+}_{\leftof f} \) under \( \rho \) \textemdash is Turing equivalent to the image of \( T_{\leftof f} \) under \( \rho \) and thus \( f \).  Thus, \( \kleeneO1 \) is subgeneric.  

But if \( X \in \HYP \) and \( X \Tleq \kleeneO1 \) then \textemdash as \( X \Tleq \zeron{\alpha} \) for some \( \alpha \in \kleeneO \) \textemdash the subgenericity of \( \kleeneO1 \) ensures that \( X \Tleq \zeron{0} \).  Thus, \( \kleeneO1 \) is a path through \( \kleeneO \) which doesn't compute any non-computable hyperarithmetic sets.  
\end{proof}

\section{Towers of Trees}\label{sec:towers}

\Cref{thm:standard,thm:non-standard} both ask ask us to produce computable trees whose paths behave genericish.  The need to keep the tree computable prevents us from using a normal forcing construction but directly building such a tree would require an excruciatingly complicated construction.  Harrington's trick to avoid a complicated worker style argument is to use jump inversion to divide up the work into levels.  

To motivate this idea, consider the proof that there are  r.e. sets in \( \highN{n+1} - \highN{n} \) or  \( \lowN{n+1} - \lowN{n} \) for any \( n \).  One could certainly construct such sets via a direct \( \zeron{n+1} \) priority argument but a far easier proof is to start with the \( n \)-th jump you want realized and apply jump inversion \( n \)-times.  However, this approach is generally limited to constructing r.e. or  \( \deltazn{2} \) sets (or, by relativization, \REA sets) since only the Sacks and Shoenfield jump inversion theorems \cite{sacks_recursive_1963,shoenfield_degrees_1959} produce a more definable set than the set being inverted  \textemdash absent the need to produced a highly definable set direct forcing/coding constructions are usually superior.  However, we evade these usual limitations because it is \( T_0 \) we need to make highly definable while it is the paths through \( T_0 \) whose jumps we want to control.

To apply this jump inversion based approach to the construction of a tree we introduce the idea of a tower of trees. The idea here is that  every path through \( T_0 \) is produced by jump inversion of some path through some tree \( T_1 \Tleq \zeroj \) and, in turn, every path through \( T_1 \)  is produced by jump inversion of some path through some tree \( T_2 \Tleq \zerojj \) and so on.   Thus, we can think of paths through \( T_{\beta} \Tleq \zeron{\beta} \) as (when joined with \( \zeron{\beta} \)) as the \( \beta \) jump of a path through \( T_0 \).  In the rest of this section, we make this idea precise and show that we can control the necessary facts about \( \jumpn{g}{\beta} \) for \( g \in [T_0] \) by controlling \( T_{\beta} \).              

\subsection{Defining Towers} 

For this definition, recall that a path in \( \kleeneO \) is a downward closed linearly ordered subset of \( \kleeneO \), i.e., \( \set{\beta \Oleq \alpha} \), \( \set{\beta \Oless \alpha} \) or a path through \( \kleeneO \) like  \( \kleeneO1 \).

\begin{definition}\label{def:tower}
A tower (of trees) consists of a non-empty path  \( I \) in \( \kleeneO \), trees \( T_\beta \) for each \( \beta \in I \) and a function \( \Gamma^{\gamma}_{\beta}(\cdot) \) for each pair \( \gamma \Oless \beta \in I \) such that 
\begin{enumerate}
    \item  \( T_\beta \Tleq \zeron{\beta} \) uniformly in \( \beta \) and \( l^{\gamma}_\beta \) is a computable function of \( \gamma, \beta \)  
    \item  \( \Gamma^{\gamma}_{\beta}(\sigma) \eqdef \Gamma^{\gamma}_{\beta}(\zeron{\gamma}; \sigma)  \) where \( \Gamma^{\gamma}_{\beta} \) is a computable functional \textemdash to disambiguate the \textit{function} \( \Gamma^{\gamma}_{\beta}(\sigma) \) is \( \zeron{\gamma} \) computable the \textit{functional} is computable.    

    \item \( \Gamma^{\gamma}_{\beta}(\cdot)  \) is an expansionary function on strings mapping \( T_{\gamma} \) to a subset of \( T_\beta \) with \( \Gamma^{\gamma}_{\gamma} \eqdef \id\restr{T_\gamma} \).   

    \item\label{def:tower:homo}   \( \Gamma^{\gamma}_{\beta}(\cdot) \) is a homeomorphism of \( [T_\gamma] \) and \( [T_\beta] \). 

    \item If \( \gamma' \in (\beta, \gamma) \) then \( \Gamma^{\gamma}_{\beta} = \Gamma^{\gamma'}_{\beta} \circ \Gamma^{\gamma}_{\gamma'}   \). 

\end{enumerate}  
\end{definition}

We will call the bottom tree in a tower the root of that tower and will say a tower has height \( \alpha \in \kleeneO \) (or \( \Ohgt{\alpha} \)) if \( I = \set{\beta \Oleq \alpha} \) and height \( < \alpha \) (or \( < \Ohgt{\alpha} \))  if  \( I = \set{\beta \Oless \alpha} \).  As there is no possibility of confusion,  we use height \( \wck \) and \( < \wck \) interchangeably.  This notation is convenient since we won't make any use of towers of height \( < \alpha \) except when \( \alpha = \wck \).   

\begin{definition}\label{def:g-beta}
Given a tower of trees with root \( T_0 \)  we define \( g^{\beta} \) to be the unique path in \( T_{\beta} \) with \(  \Gamma^{\beta}_{0}(g^{\beta}) = g = g^{0} \) where \( g \in [T_0] \).  
\end{definition}

The similarity between \( g^{\beta} \) and \( \jumpn{g}{\beta} \) is deliberate as \textemdash in the subgeneric towers we are interested in \textemdash  we will think of \( g^{\beta} \) as a version of \( \jumpn{g}{\beta}  \).  Since  \( \Gamma^{\beta}_{0} \) is a homeomorphism  all paths through \( T_\beta \) are of the form \( g^{\beta} \) for some \( \beta \).

\subsection{Subgeneric Towers}

  To use a tower of trees to produce trees whose paths are uniformly \( \alpha \)-subgeneric we need to ensure the tower has a number of extra properties.  If we merely wanted to build paths realizing the least possible \( \alpha \) jump it would be enough to ensure that all paths on \( T_{\beta} \) are \( \beta \)-generic in the local forcing on \( T_{\beta} \).  We can also ensure \( \jumpn{g}{\beta} \) doesn't compute any other paths by controlling \( T_{\beta} \).

However, ensuring that  \( g \meet \zeron{\beta} = \Tzero \) requires\footnote{Obviously, we could theoretically directly construct \( T_0 \) to have all the desired properties so it is only required given our desire to leave satisfying \( \zeron{\beta} \) properties to \( T_{\beta} \).} cooperation between the levels of the tower \textemdash otherwise we might have \( T_0 \) coding \( \zeroj \) into every path.  We address this by ensuring that \( T_0 \) gives \( T_1 \) the chance to choose a path whose image under  \( \recfnl{i}{}{} \) disagrees with any particular set \textemdash and \( T_1 \) does the same for \( T_2 \) and so on.  To help make this notion precise we offer the following definition.

\begin{definition}\label{def:splitting-path}
Given a functional or monotone function on strings \( \zeta \), a tree \( T \subset \wstrs \) and a path \( f \in [T] \) we say that \( T \) is \( \zeta \) splitting over \( f \) if every initial segment \( \sigma  \) of  \(  f \)  is extended by \( \zeta \)-splitting strings \( \tau_0, \tau_1  \) in \( T \) \textemdash where \( \tau_0, \tau_1 \) are  \( \zeta \) splitting (or \( i \)-splitting if \( \zeta \) has index \( i \)) if \( \zeta(\tau_0) \incompat \zeta(\tau_1) \). 
\end{definition}

We now specify the properties that will ensure that \( [T_0] \) is uniformly \( \alpha \)-subgeneric when it is part of a tower of height \( \alpha \).   In what follows we use \( I^{-} \) to denote the set of non-maximal elements in \( I \).

\begin{definition}\label{def:subgeneric-tower}
A tower which satisfies the following features for all  \( \beta \Oless \lambda \in I^{-} \) is called a subgeneric tower.

\begin{enumerate}
    \item\label{def:subgeneric-tower:jump} \( \displaystyle \zeron{\beta \Oadd 1} \Tplus g^{\beta \Oadd 1} \Tgeq \jump{\left(\zeron{\beta} \Tplus g^{\beta} \right)}  \)  \textemdash uniformly\footnote{By this we mean that there is a single functional that works for all paths in the relevant trees witnessing each reduction and that an index for the reduction is computable from \( \beta \).} in \( \beta \) and \( g^{\beta} \)

    \item\label{def:subgeneric-tower:non-trivial}  \(\displaystyle g^{\beta} \nTleq \zeron{\beta}  \)

    \item\label{def:subgeneric-tower:incompat}  \( \displaystyle g^{\beta} \neq h^{\beta}  \implies  g^{\beta} \nTleq \zeron{\beta} \Tplus  h^{\beta} \)

    \item\label{def:subgeneric-tower:splitting-subtree} Let \( T^{\alpha}_\beta \subset T_\beta \) be the image of \( T_{\alpha} \) under \( \Gamma^{\alpha}_\beta \).  If \( \Theta \) is a \( \zeron{\beta} \) computable functional and  \( T^{\alpha}_\beta \) is \( \Theta \) splitting over \( f^{\beta} \) for all \( \alpha \in [\beta,  \lambda) \) then \( T^{\lambda}_\beta \) is \( \Theta \) splitting over \( f^{\beta} \).      

    \item\label{def:subgeneric-tower:genericity}  If \( W \subset \wstrs \) is r.e. in \( \zeron{\beta} \) then all \( g^{\beta} \in [T_\beta] \) either meet \( W \)  or strongly avoid \( W \)  on \( T_\beta \), i.e.,   
    \[  \exists(\sigma \subfun g^{\beta})\bigl[\sigma \in W \lor \forall(\tau \supfun \sigma)\left(\tau \nin W \isect T_{\beta}\right) \bigr] \]

\end{enumerate}

\end{definition}

Note that, in part \ref{def:subgeneric-tower:splitting-subtree}, we always have \( f^{\beta} \in [T^{\alpha}_\beta] \) since \( \Gamma^{\alpha}_\beta \) is a homeomorphism of \( [T_{\alpha}] \) with  \( [T_{\beta}] \).  As we mentioned above, in this paper, we will only make use of towers where \( I \) has a maximal element or is a full path through \( \kleeneO \)  but it does raise the following question.

\begin{question}
If \( T_0 \) is the root of a subgeneric tower of height \( < \lambda \) for \( \lambda \in \kleeneO- \) must every \( g \in [T_0] \) be \( \lambda \)-subgeneric?     
\end{question} 

We hypothesize that \textemdash in contrast to the fact that any real \( n \)-generic for all \( n \in \omega \) is \( \omega \)-generic \textemdash  the answer is no for every limit notation \( \lambda \).  

\subsection{From Subgeneric Towers To Subgenerics}

It is finally time to connect subgeneric towers and subgenerics.   

\begin{lemma}\label{prop:subgeneric-tower-properties}
If \( T_0 \) is the root of a subgeneric tower of height at least \( \kappa \) then \( [T_0] \) is uniformly \( \kappa \)-subgeneric \textemdash including when \( \kappa = \wck \).      
\end{lemma}

We first identify a property sufficient to guarantee the truth of this claim.

\begin{lemma}\label{lem:subgeneric-criteria}
If \( T_\beta, \Gamma, I \) is a subgeneric tower of height \( \kappa \)  satisfying the following for all \( \beta \Oless \alpha \in I \) and  \( g, h \in [T_0] \)
\begin{enumerate}
    \item\label{prop:subgeneric-tower-properties:jump}   \( \displaystyle \jumpn{g}{\alpha} \Tequiv g \Tplus \zeron{\alpha} \Tequiv g^{\alpha} \Tplus \zeron{\alpha} \)  \textemdash uniformly in \( \alpha, g \)

    \item\label{prop:subgeneric-tower-properties:subg} \( \displaystyle X \Tleq \zeron{\alpha}  \land  X \Tleq \jumpn{g}{\beta}  \implies   X \Tleq \zeron{\beta} \) 

    \item\label{prop:subgeneric-tower-properties:incompat}  \( \displaystyle  g \nTleq \jumpn{h}{\beta}  \)      

\end{enumerate}
 then \( T_0 \) is uniformly \( \kappa \)-subgeneric \textemdash including when \( \kappa = \wck \).   
\end{lemma}

\begin{proof}
To see that \( g \) is \( \alpha \)-subgeneric for \( \alpha \in I \)  observe that by part \ref{def:subgeneric-tower:non-trivial} of \cref{def:subgeneric-tower} we already know that \( g^{\beta} \nTleq \zeron{\beta} \) and the first claim entails that \( g \nTleq \zeron{\beta} \) and the first two claims above correspond to the other two requirements for \( \alpha \)-subgenericity.  The uniformity required for \( T_0 \) to be uniformly \( \alpha \)-subgeneric is guaranteed by the first claim and the only other property we need is the last claim above.  This is enough to verify the lemma when \( \kappa < \wck \) and if \( \kappa = \wck \) then \( I \) is a path through \( \kleeneO \) and the  above argument shows we satisfy the definition of uniform \( \alpha \)-subgenericity for all \( \alpha \in I  \) \textemdash just as uniform \( \wck \)-subgenericity requires.  
\end{proof}

We now prove \cref{prop:subgeneric-tower-properties} by verifying the properties from \cref{lem:subgeneric-criteria} hold for any subgeneric tower of height \( \kappa \).

\begin{proof}
We first observe that the third claim in \cref{lem:subgeneric-criteria}  is immediate from the first claim and the fact that \( g^{\beta} \nTleq h^{\beta} \Tplus \zeron{\beta} \).  Thus, it is enough to prove the first two claims.

To prove claim (\ref{prop:subgeneric-tower-properties:jump}) suppose, for the sake of contradiction, that \( \alpha \) is a \( \Oless \) minimal counterexample.  If \( \alpha  \) was a successor then the claim would follow from the inductive assumption plus part \ref{def:subgeneric-tower:jump} of \cref{def:subgeneric-tower}.   So suppose that \( \alpha \) is a limit notation.  As \( \Gamma^{\alpha}_{\Olim{\alpha}{n}}  \) is uniformly \(  \zeron{\alpha} \) computable we have \( \zeron{\alpha} \Tplus g^{\alpha} \Tgeq \jumpn{g}{\Olim{\alpha}{n}} \) uniformly.  As \( \jumpn{g}{\alpha} = \Union_{n \in \omega} \jumpn{g}{\Olim{\alpha}{n}} \) this guarantees that \(  \zeron{\alpha} \Tplus g^{\alpha} \Tgeq \jumpn{g}{\alpha} \).  The other direction is trivial and the equivalence of \( \zeron{\alpha} \Tplus g^{\alpha}  \) with \( \zeron{\alpha} \Tplus g  \) is immediate from the fact that \( \Gamma^{\alpha}_{0} \Tleq \zeron{\alpha} \).  This yields the desired equivalence and the uniformity follows from the uniformity of the computations used in this proof.

This leaves only claim (\ref{prop:subgeneric-tower-properties:subg}) to prove.  Assume, by way of induction, that whenever \( \beta \Oleq \gamma \Oless \alpha \) and \( X  \)  is computable in both \(  \zeron{\gamma} \) and \( \jumpn{g}{\beta} \) then \( X \Tleq \zeron{\beta} \) \textemdash allowing \( \beta = \gamma \) makes no difference as this case is trivial.  We suppose \( X \) is computable in both \( \zeron{\alpha} \) and  \( \jumpn{g}{\beta} \) and prove \( X \Tleq \zeron{\beta} \).  

Since \( \zeron{\beta} \Tplus g^\beta \Tequiv  \jumpn{g}{\beta}  \) there is some \( \zeron{\beta} \) computable functional \( \zeta \) with \( \zeta(g^{\beta}) = X \).  First, suppose there is some \( \gamma \in [\beta, \alpha) \) with \( T^{\gamma}_\beta \) (the image of \( T_{\gamma} \) under \( \Gamma^{\gamma}_\beta \)) is not \( \zeta \) splitting over \( g^{\beta} \).  In this case, we can compute \( X \) from \( \zeron{\gamma} \) as there is some \( \sigma \subfun g^{\beta} \) lacking any extension \( \tau \in T^{\gamma}_\beta \Tleq \zeron{\gamma} \) with \( \zeta(\tau) \) incompatible with \( X \).  Thus, by the inductive hypothesis, we can infer that \( X \Tleq \zeron{\gamma} \Tleq \zeron{\alpha} \). 

Thus, we can assume, that for every \( \gamma \in [\beta, \alpha) \) the tree \( T^{\gamma}_\beta \) is \( \zeta \) splitting over \( g^{\beta} \).  By part \ref{def:subgeneric-tower:splitting-subtree} of the definition of a subgeneric tower it follows that \( T^{\alpha}_\beta \) is \( \zeta \) splitting over \( g^{\beta} \).  Thus, the \( \zeron{\alpha} \) computable set  
\[ \set{\sigma}{\sigma \in T_{\alpha} \land \zeta\left(\Gamma^{\alpha}_{\beta}(\sigma)\right) \incompat X } \]          
isn't strongly avoided by \( g^{\alpha} \) on \( T_{\alpha} \) so \textemdash by the genericity requirement (part \ref{def:subgeneric-tower:genericity}) for a subgeneric tower \textemdash \( g^{\alpha} \) must meet this set.  Thus, \( \zeta(g^{\beta}) =  \zeta \mathbin{\circ} \Gamma^{\alpha}_{\beta}(g^{\alpha}) \neq X \) contradicting our assumption and completing our proof.
\end{proof}

This lemma reduces proving \cref{thm:standard} and \cref{thm:non-standard} to the construction of subgeneric towers with the appropriate properties.  In \cref{sec:standard-tower}  we will prove \cref{thm:standard}  by constructing subgeneric towers of any height \( \alpha \in \kleeneO \).  In \cref{sec:non-standard-tower} we apply the construction in \cref{sec:standard-tower}  to a `non-standard' recursive ordinal to build a fully subgeneric tower and prove \cref{thm:non-standard}.  We will therefore sometimes call subgeneric towers of height \( \alpha \in \kleeneO \) standard subgeneric towers and subgeneric towers of height \( \wck \) as non-standard subgeneric towers\footnote{Morally speaking, the non-standard subgeneric tower is the result of applying the construction to the `non-standard' recursive ordinal but only the well-founded initial segment of this object is a subgeneric tower as we must have \( I \subset \kleeneO \).}.

\section{Constructing Subgeneric Towers}\label{sec:standard-tower}

In this section, we construct a subgeneric tower of any height \( \alpha \in \kleeneO \).   Our strategy for building a subgeneric tower of height \( \alpha \) will be to start with \( T_{\alpha} \) and then define every \( T_{\beta}, \beta \Oless \alpha \) from \( T_{\alpha} \).  We will build \( T_{\beta} \) as a kind of garbled copy of  \( T_{\beta \Oadd 1} \).  Assuming we knew how to do this correctly we could  build finite height subgeneric towers but how can we extend this approach to limit levels?  If \( T_n \) is defined by (in some sense) copying \( T_{n+1} \) how can we also ensure \( T_n \) has the `same' paths as \( T_{\omega} \)?  

The basic idea is that we build \( T_{n} \) to satisfy \( T_{n}\restr{n} = T_{\omega}\restr{n}  \).  Since we are insisting that \( T_{n+1}\restr{n+1} = T_{\omega}\restr{n+1}  \) we will have  \( T_{n}\restr{n} = T_{n+1}\restr{n} = T_{\omega}\restr{n}  \) and we will build \( T_n \) from \( T_{n+1} \) above each \( \sigma \in T_{\omega} \) of length  \( n  \).  This turns out to completely define each \( T_n \) since we can project strings of length \( l \) down from \( T_{l}\restr{l} = T_{\omega}\restr{l}  \) and will guarantee that each \( [T_n] \) is homeomorphic to \( T_{\omega} \) \textemdash provided we construct \( T_n \) from \( T_{n+1} \) appropriately.         

Extending this idea beyond \( \omega \) requires a bit of attention to detail to ensure that the demands imposed by all the different limit ordinals are both compatible and give rise to a uniquely defined tower \textemdash and we will explore this in \cref{ssec:build-std-tower} before using the recursion theorem to simultaneously construct the whole subgeneric tower.  However, this leaves us with the construction of \( T_{\beta} \) from \( T_{\beta \Oadd 1} \) as the primary barrier we face.  We tackle that barrier now by proving the key lemma of the paper.

\subsection{Pulldown Lemma}\label{sec:pulldown-lemma}

Recall that our construction of \( T_\beta \) from \( T_{\beta \Oadd 1} \) will need to pass along to \( T_{\beta \Oadd 1} \) strings which produce disagreeing computations for  \( \zeron{\beta} \) computable functionals \textemdash part \ref{def:subgeneric-tower:splitting-subtree} of \cref{def:subgeneric-tower}.    This will be easier to do if we assume that infinitely often we know that both \( \sigma\concat[0] \) and \( \sigma\concat[1] \) are in   \( T_{\beta \Oadd 1} \) so we can map them to the  disagreeing computations.

\begin{definition}\label{def:nice-tree}
A tree \( T  \) is nice if whenever \( \lh{\sigma} \equiv 3 \pmod{4} \) then \( \sigma \in T \) iff \( \sigma^{-} \in T \) and \( \sigma = \sigma^{-}\concat[0] \) or \( \sigma = \sigma^{-}\concat[1] \).   
\end{definition} 

We can now state the main combinatorial lemma for this section.

\begin{lemma}\label{lem:tree-minus-one} 
Given a nice tree \( T' \Tleq \jump{X} \) and \( l \in \omega \) such that \( T'\restr{l} \Tleq X \) there is a tree \( T \) and a function \( \Gamma \Tleq \jump{X} \) such that   
\begin{enumerate}
    \item\label{lem:tree-minus-one:comp} \(  T \) is uniformly computable in \( X \)
    \item\label{lem:tree-minus-one:nice}   \( T \) is nice.
    \item\label{lem:tree-minus-one:copy}  \( T\restr{l} = T'\restr{l} \) and \( \Gamma \) is the identity on \( T'\restr{l} \).  
    \item\label{lem:tree-minus-one:expansionary} \( \Gamma \) is expansionary, indeed \( \Gamma(\sigma\concat[x]) \supfun \Gamma(\sigma)\concat[x] \) whenever \( \Gamma(\sigma\concat[x])\conv  \).   
    \item\label{lem:tree-minus-one:dom-zeta} \( \dom \Gamma = T' \) 
    \item\label{lem:tree-minus-one:homo} \( \Gamma \) is a homeomorphism of \( [T'] \) and \( [T] \).
    \item\label{lem:tree-minus-one:subg}  If \( g \in [T] \) then \( g \) either meets or strongly avoids on \( T \) every \( X \)-r.e.  set of strings.   
    \item\label{lem:tree-minus-one:jump-inv} If \( f \in [T] \) then (uniformly)  \( \jump{(f \Tplus X)} \Tequiv f \Tplus \jump{X} \Tequiv \hat{f} \Tplus \jump{X} \) where \( \hat{f} \) is the unique path through \( [T'] \) such that \( \Gamma(\hat{f}) = f \).   
    \item\label{lem:tree-minus-one:path-noncomp} \( g \in [T] \implies g \nTleq X \) 
    \item\label{lem:tree-minus-one:pair-noncomp}  \( g, h \in [T] \land g \neq h \implies g \nTgeq h \Tplus X \) 
    \item\label{lem:tree-minus-one:splitting} If \( 4i + 2 > l \),  \( \sigma \in T' \) and \( \lh{\sigma} = 4i + 2 \) then either \( \Gamma(\sigma\concat[0]), \Gamma(\sigma\concat[1]) \) are \( i \)-splitting or no extensions of \( \Gamma(\sigma) \) in \( T \) are \( i \)-splitting.              
\end{enumerate} 
Indexes for \( T \) and \( \Gamma \)  are computable from \( l \) and indexes \( e, e_0  \) for \( T' \) and \( T'\restr{l} \) respectively.  Moreover, if \( e_0 \) is a code for a nice \( X \)-computable tree \( T'\restr{l} \) of height at most \( l \) and we define \( T' \) to be the largest nice tree\footnote{But for the need to be nice, it would be enough to say \( \sigma \) is in \( T' \) iff every \( \sigma' \subfun \sigma \) is either in \( T'\restr{l} \) or of length greater than \( l \) and  \( \recfnl{e}{\jump{X}}{\sigma}\conv = 1 \).  We can ensure niceness by also removing all strings \( \sigma \) such that for some \( x \) with \( 2x+1 > l \),  \( \sigma\restr{2x+1} \neq \sigma^{-}\concat[i], i \leq 1 \) or \( \sigma^{-}\concat[1 -1] \nin T' \).} extending \( T'\restr{l} \) satisfying \( \sigma \in T' \setminus T'\restr{l} \implies \recfnl{e}{\jump{X}}{\sigma}\conv = 1 \) then \( T, \Gamma \) (as constructed from \( e_0, e, l \)) still satisfy the statement of the lemma subject to the caveat that \( \Gamma \) may merely be r.e. in \( \jump{X} \) if \( \recfnl{e}{\jump{X}}{} \)  is partial.

\end{lemma}

 In the moreover claim, we are showing that the construction fails gracefully even when given a badly behaved index \textemdash a fact we will need later to ensure that when we use the recursion theorem to build a subgeneric tower we avoid trivial solutions.  Note that we will always have  \( T' = \set{\sigma}{\recfnl{e}{\jump{X}}{\sigma}\conv = 1} \) whenever that set is a nice tree extending  the \( X \)-computable tree with index \( e_0 \) we denoted by \( T'\restr{l} \).

\subsection{Construction Overview}

 The construction will proceed in stages and we will build \( \Gamma, T \) as the limits of \( X  \) computable sequences of stagewise approximations \( \Gamma_s, T_s \).  Once enumerated into \( T \) strings will never leave so \( T_s \) will always be the set of strings enumerate before stage \( s \).  We will ensure that \( T \) is \( X \) computable whenever \( T'\restr{l} \Tleq X \)  by permanently excluding any \( \tau \) from \( T \) that isn't enumerated by the end of stage \( \godelnum{\tau} \), i.e., \( \tau \in T \iff \tau \in T_{\godelnum{\tau} + 1} \).  We ensure that \( \rng \Gamma \subset T \) by enumerating every \( \tau' \subfun \tau \) into \( T\) whenever we set  \( \Gamma_s(\sigma) = \tau \).  Since \( T'\restr{l} \Tleq X \) we can also stipulate that \( \Gamma_s \) is the identity on \( T'\restr{l}  \)  and undefined on all strings of length at most \( n \) not in \( T'\restr{l} \).       

 At the outset of the construction we fix an \( X \) computable approximation \( T'_s \)  to \( T' \) such that \( T'_s \) is always finite, downward closed and \( \lim_{s \to \infty} T'_s(\sigma) \) exists  whenever \( \recfnl{e}{\jump{X}}{\sigma}\conv  \) (where we assume this is a \( 0,1 \) valued function).  We can also assume that our approximations are always correct on strings of length at most \( l \), i.e., \( T\restr{l} = T_s\restr{l} = T'_s\restr{l} = T'\restr{l}  \).   \( \Gamma_s(\sigma) \) is the identity on \( T'\restr{l}  \) and undefined on all strings of length at most \( n \) not in \( T'\restr{l} \).       

\subsubsection{Niceness}\label{ssec:niceness}

We need to build \( T \) to be nice.  This is relatively trivial to do simply by placing \( \sigma\concat[0], \sigma\concat[1] \) into \( T \) whenever we place \( \sigma \) into \( T \) and \( \lh{\sigma} = 4k + 2 \geq l \) \textemdash since \( T'\restr{l} \) is also assumed to be nice the constraint that \( T\restr{l} = T'\restr{l} \) won't be an issue.   However, the construction will be complicated enough without breaking up every part of the construction into two cases depending on whether the length mod \( 4 \) of strings in   \( T \).  Therefore, we simply assume that the construction we give below is (essentially) describing what happens on strings whose length isn't \( 3 \pmod{4} \) and what happens on those strings is filled in behind the scenes with \( \Gamma \) always following the \( 0 \) branch by default.  Indeed, the only reason we can't just modify \( T \) to be nice  after the fact is that we need to include these strings in \( T_s \) when we ask if there is some way to extend \( \Gamma_{s}(\tau) \) to have some desired property.

\subsubsection{Requirements}

To understand the construction it will help to keep in mind that if \( g \in [T'] \) then we will end up defining  
\[ \Gamma(g) = \sigma_0\concat[g(0)]\concat\sigma_1\concat[g(1)]\concat\sigma_2 \ldots \concat\sigma_{k}\concat[g(k)]\concat\sigma_{k+1} \ldots  \] 
where the strings \( \sigma_i \) are the results of letting  \( \Gamma_{s+1}(\tau) \) properly extend  \( \Gamma_s(\tau) \) when we see some string \( \sigma' \supfun  \Gamma_s(\tau) \) in \( T_s \) which will help us meet various  requirements, e.g., trying to meet \( X \)-r.e. sets of strings.   We now list those requirements.

\begin{requirements}
\require{G}{i} \exists(\sigma \subfun g)\left[ \Gamma(\sigma) \supfun \REset(X){i} \lor \forall({\tau \supfun \Gamma(\sigma)})\left(\tau \nin T \lor \tau \nsupfun \REset(X){i}  \right)  \right] \\[.5em]
\require{S}{i} \begin{aligned}
(\exists \sigma &\subfun g)\Bigl[\recfnl{i}{\Gamma(\sigma\concat[0]) \Tplus X}{} \incompat \recfnl{i}{\Gamma(\sigma\concat[1]) \Tplus X}{} \lor \phantom{X} \\
& \forall({\tau_0, \tau_1 \supfun \Gamma(\sigma)})\bigl(\tau_0, \tau_1 \in T \implies \recfnl{i}{\tau_0 \Tplus X}{} \compat \recfnl{i}{\tau_1 \Tplus X}{}   \bigr) \Bigr]
\end{aligned} \\[.5em]
\require{R}{i}  g \neq f \implies \recfnl{i}{X \Tplus \Gamma(f)}{}\diverge  \lor \recfnl{i}{X \Tplus \Gamma(f)}{} \incompat \Gamma(g)
\end{requirements}

When defined \( \Gamma(\sigma) \) will be tasked with meeting a single requirement determined by the length of \( \sigma \).  Even though we only work to meet a single requirement for each string \( \sigma \) the construction has the form of a finite injury argument since anytime we redefine \( \Gamma \) on some substring of \( \sigma \) we are forced to abandon our definition of \( \Gamma(\sigma) \).

\subsection{Stagewise Operation}

 We break each stage up into two parts.  In the first part, we try to satisfy the explicit requirements by searching for appropriate extensions.  In the second part we work satisfy the implicit requirement that \( \dom \Gamma = T' \).  As we can only consider finitely many strings we say that \( \sigma \) gets attention at stage \( s = \godelnum{\tau} \) if \( \lh{\sigma} > l \) and \( \Gamma_s(\sigma)\conv \subfun \tau \).  Note that if \( \tau \) fails to enter \( T \) during this stage it is permanently excluded.

\subsubsection{First Half of Stage}\label{ssec:first-half-stage}

In the first half of the stage we start with the shortest string \( \sigma \)  receiving attention and execute the first case below which applies to \( \sigma \).  As directed below, we then repeat this process in order of increasing length with each string receiving attention until we either run out of strings or set \( \Gamma_{s+1}(\sigma) \neq \Gamma_s(\sigma) \) for some \( \sigma \) at which point we move on to the second half of the stage.

In all cases,  when we say set \( \Gamma_{s+1}(\sigma) \) to extend \( \nu \) we mean to define  \( \Gamma_{s+1}(\sigma) - \nu'\concat[k] \) where \( \nu' \) is a maximal extension of \( \nu  \) in \( T_s \) and \( k \) is chosen large.  Choosing \( k \) large  ensures that we haven't already committed to keeping  \( \Gamma_{s+1}(\sigma) \) \textemdash or any of it's extensions \textemdash out of \( T \).

\begin{pfcases*}
\case[\( \sigma \nin T'_{s+1} \)] Set \( \Gamma_{s+1}(\sigma) = \diverge \) and move on to the second half of the stage.

\case[\(\lh{\sigma} = 4i   \)] (\req{R}{i})  We first check if we need to assist any higher priority (here meaning permanently entering \( T' \) before \( \sigma \)) strings \( \upsilon \)  trying to meet some \req{R}{i'} requirement.  We can't know for sure when strings enter \( T' \) but we can say that \( \upsilon \) appears to enter \( T' \) at stage \( t \) if \( t \) is the least stage such that \( \tau \in T'_{s'} \) for all \( s' \in [t, s] \).    Now let \( \upsilon \in T'_s \) be the string (if any) which appears to enter \( T' \)  at the least stage  and before \( \sigma \) appears to enter \( T' \) such  that for some \( i' \),  \( \lh{\upsilon}= 4i'  > l  \) and 
\begin{align}
\lnot &\Psi(\Gamma_{s}(\sigma), \upsilon) \land \exists(\tau \supfun \Gamma_{s}(\sigma))\left(\tau \in T_s \land \Psi(\tau, \upsilon)\right) \label{eq:help-upsilon} \\
& \text{where } \Psi(\tau, \upsilon) \iffdef  \exists({\upsilon^{*} \supfun \Gamma_{s}(\upsilon)\conv}) \left(\upsilon^{*} \in T_s \land  \recfnl[s]{i'}{\upsilon^{*} \Tplus X}{} \incompat \tau\right) \nonumber 
\end{align} 
If there is such an \( \upsilon \)  let \( \tau \supfun \Gamma_{s}(\sigma) \) satisfy \( \Psi(\tau, \upsilon)  \) and set \( \Gamma_{s+1}(\sigma) \) to an extension of \( \tau \) and continue on to the second half of the stage.  Otherwise, check if 
\begin{equation*}
\begin{split}
\exists(\xi \in T'_s)\Bigl(&\xi \incompat \sigma \land \Gamma_s(\xi)\conv \supfun \recfnl[s]{i}{\Gamma_s(\sigma) \Tplus X}{} \supfunneq \Gamma_s(\xi^{-})  \land \\  
                            &\exists(\tau \supfun \Gamma_{s}(\sigma))\left(\tau \in T_s \land \recfnl[s]{i}{\tau \Tplus X}{} \incompat \xi \right)\Bigr)
\end{split}
\end{equation*}   
If so set \( \Gamma_{s+1}(\sigma) \) to an extension of \( \tau \) where \( \tau \)  and continue on to the second half of the stage otherwise do nothing. 

\case[\( \lh{\sigma} = 4i + 1 \)] (\req{G}{i})  If \( \Gamma_{s}(\sigma) \supfun \REset[s](X){i}  \) do nothing.  Otherwise check if there is some  \( \nu  \supfun \Gamma_{s}(\sigma) \) in \( T_s \isect \REset[s](X){i} \).  If so set \( \Gamma_{s+1}(\sigma) \) to extend \( \nu \). 

\case[\( \lh{\sigma} = 4i + 2 \)] (\req{S}{i})  Check if there are \( i \)-splitting strings (converging in \( s \)-steps)  \( \tau_0 \leftof \tau_1 \in T_s \) extending \( \Gamma_s(\sigma) \).    If not or \( \Gamma_s(\sigma\concat[0]) \) and \( \Gamma_s(\sigma\concat[1]) \) are already seen to be \( i \)-splitting in \( s \)-steps do nothing.  If so, set \( \Gamma_{s+1}(\sigma) \) equal to the longest common initial segment of \( \tau_0 \) and \( \tau_1 \) (to maintain \( \Gamma \) as an expansionary function) and set  \( \Gamma_{s+1}(\sigma\concat[j]) \) to extend \( \tau_j \).

\end{pfcases*}

\subsubsection{Second Half of Stage}

Let \( \sigma \) be the \( \subfun \) maximal element such that \( \Gamma_{s+1}(\sigma)\conv \subfunneq \tau \) and \( \lh{\sigma} \geq l \) and let \( x \) be such that \(  \Gamma_{s+1}(\sigma)\concat[x] \subfun \tau \).    If there is no such \( \sigma \)  end the stage.   

If \( \Gamma_{s+1}(\sigma)\concat[x] = \tau \) and \( \lh{\sigma} \equiv 2 \pmod{4} \) or \( x \leq 1 \) then enumerate \( \tau \) into \( T \).  This will ensure that even if we don't see \( \sigma\concat[x] \) enter \( T' \) until much later we haven't committed to keeping \( \Gamma_{s+1}(\sigma)\concat[x] \) out of \( T \).  

Finally, if \( \sigma\concat[x] \in T'_s \) but either \( \Gamma_{s}(\sigma\concat[x])\diverge \) or we explicitly marked \(  \Gamma_{s+1}(\sigma\concat[x]) \) as undefined earlier in the stage \textemdash possibly because \( \Gamma_{s+1}(\sigma) \neq \Gamma_s(\sigma) \) \textemdash set \( \Gamma_{s+1}(\sigma\concat[x]) = \tau\concat[k] \) for some large \( k \).  Note that, in this case we explicitly \textbf{don't} set  \( \Gamma_{s+1}(\sigma\concat[x]) \) to extend some maximal element in \( T_s \) extending \( \tau \).

\subsection{Verification}

We now argue that this construction vindicates  \cref{lem:tree-minus-one}.  We start with a utility lemma

\begin{lemma}\label{lem:inf-attention}
If \( \Gamma(\sigma)\conv \) then \( \sigma \) receives attention infinitely often.
\end{lemma}
\begin{proof}
There is some stage \( s_0 \) at which  \( \Gamma_s(\sigma) \) settles down to equal some \( \tau \).  There are infinitely many \( \tau' \supfun \tau \) so there are infinitely stages \( s > s_0 \)  where \( s = \godelnum{\tau'} \) with \( \tau' \supfun \tau \).  
\end{proof}

With this result in hand, we note that we can assume that while \( \Gamma_s(\sigma)\conv \) we delay any stage where \( T'_{s+1}(\sigma) \neq T'_s(\sigma) \) until a stage at which \( \sigma \) receives attention.  By the argument above, there will always be such a stage so there is no harm in pausing our approximation until we see such a stage.  This has the same effect as if the construction had undefined \( \Gamma_{s+1}(\sigma) \) whenever \( \sigma \) had left \( T' \) since the last stage it received attention.  The benefit of redefining our approximation is simply to avoid having to repeatedly specify that \( \sigma \) has or hasn't been out of \( T'_s \) since last receiving attention.   We now observe that we try and define \( \Gamma(\sigma) \) when \( \sigma \in T' \).   

\begin{lemma}\label{lem:almost-always-defined}
If \( \sigma \in T' \) and \( \Gamma(\sigma^{-})\conv \)  then for almost all stages \( \Gamma_s(\sigma)\conv \).
\end{lemma}
\begin{proof}
This is trivial if \( \lh{\sigma} \leq l \) so we may assume that \( \lh{\sigma} > l \).  The  only ways that we can set \( \Gamma_{s+1}(\sigma)\diverge \) is when \( \sigma \nin T'_s \) or \( \Gamma_{s+1} \) changes on some proper initial segment of \( \sigma \).  The assumption ensures the later eventually stops happening and since \( \sigma \in T \) we also have \( \sigma \in T'_s \) for almost all \( s \).  If \( s = \godelnum{\tau} \) with \( \tau \supfun \Gamma_s(\sigma^{-})\concat[x] \) where \( x = \sigma(\lh{\sigma}-1) \) and \( \sigma \in T'_s \) then we define \( \Gamma_{s+1}(\sigma)  \).  
\end{proof}

We now argue that each requirement imposes finitely much injury. 

\begin{lemma}\label{lem:req-gi-converges}
Suppose that \( \sigma \in T' \), \( \Gamma(\sigma^{-})\conv \)  and \( \lh{\sigma} = 4i + 1 > l \) then \( \Gamma(\sigma)\conv \).  Moreover, \( \Gamma(\sigma) \) witnesses the satisfaction of \req{G}{i}  
\end{lemma}
\begin{proof}
By \cref{lem:almost-always-defined} it is enough to prove the moreover claim since at that point we will stop redefining \( \Gamma_{s+1}(\sigma)  \).  However, it is evident that once we change \( \Gamma_{s+1}(\sigma)  \) once to meet some \( X \)-r.e. set we never need to do so again \textemdash assuming we've already settled on \( \Gamma(\sigma^{-}) \) and on \( \sigma \in T' \).   This completes the proof since, if we don't ever change \( \Gamma_{s+1}(\sigma)  \) to meet \( \recfnl{i}{X}{} \) that's because we don't have the chance above  \( \Gamma_{s}(\sigma)  \)  and either way we satisfy \req{G}{i}.  
\end{proof}

\begin{lemma}\label{lem:req-si-converges}
Suppose that \( \sigma \in T' \), \( \Gamma(\sigma^{-})\conv \)  and \( \lh{\sigma} = 4i + 2 > l \) then \( \Gamma(\sigma)\conv \),  \( \Gamma(\sigma\concat[0])\conv \) and  \( \Gamma(\sigma\concat[1])\conv \).  Moreover, \( \Gamma(\sigma\concat[0]), \Gamma(\sigma\concat[1]) \) and \( \Gamma(\sigma) \)  witness the satisfaction of \req{S}{i}.  
\end{lemma}
\begin{proof}
Since \( T' \) is nice and \( \lh{\sigma} \equiv 2 \pmod{4} \)  we have that \( \sigma\concat[0], \sigma\concat[1] \in T' \).  Once all strings permanently enter \( T' \) and \( \Gamma_s(\sigma^{-}) \) settles then by the reasoning of \cref{lem:almost-always-defined} we eventually define all of \( \Gamma_s(\sigma), \Gamma_s(\sigma\concat[0]), \Gamma_s(\sigma\concat[1])  \).  These only change if we see some way to extend \(  \Gamma_s(\sigma\concat[0]) \) and \( \Gamma_s(\sigma\concat[1])  \) to \( i \)-split.  Either way, \( \Gamma \) settles down on all three values and we witness the satisfaction of \req{S}{i}.
\end{proof}

We don't bother with a lemma for the case where \( \lh{\sigma} =4i+3 \) as the only action we might take for such a \( \sigma \) is covered by the prior lemma so, clearly, \( \Gamma(\sigma)\conv \) whenever \( \Gamma(\sigma^{-})\conv \).  Similarly, whenever \( \lh{\sigma} \leq l \) we have \( \Gamma(\sigma)\conv \) by direct construction.  This leaves only the case where \( \lh{\sigma} \equiv 0 \pmod{4} \).    

\begin{lemma}\label{lem:req-ri-converges}
Suppose that \( \sigma \in T' \), \( \Gamma(\sigma^{-})\conv \)  and \( \lh{\sigma} = 4i > l \) then \( \Gamma(\sigma)\conv \).  Moreover, \( \Gamma(\sigma) \) witnesses the satisfaction of \req{R}{i} \textemdash that is, if \( f \in [T] \) and \( f \supfun \Gamma(\sigma) \) then \( \recfnl{i}{f \Tplus X}{} \) is either partial or disagrees with \( \Gamma(g) \) for all \( g \in [T'] \).     
\end{lemma}
We prove this lemma simultaneously with the next lemma.
\begin{lemma}\label{lem:tprime-converge}
If \( \sigma \in T' \) then \( \Gamma(\sigma)\conv \). 
\end{lemma}
\begin{proof}
Suppose, for a contradiction, that \( \sigma \in T' \) is the witness which permanently enters \( T' \) at the least stage such that \( \Gamma(\sigma)\diverge \).  By the inductive assumption and the above lemmas/discussion we can assume that \( \lh{\sigma} = 4i > l \), \( \Gamma_s(\sigma^{-}) \) has settled down to it's true value, \( \sigma \) has permanently entered \( T' \) and that \( \Gamma_s(\sigma) \) is defined.  Thus, the only way we could have \( \Gamma_s(\sigma)\diverge \) is if we act to try and meet \req{R}{i} infinitely many times.

However, since acting to meet \req{R}{i} always sets \( \Gamma_{s+1}(\sigma) \) to be a proper extension of \( \Gamma_s(\sigma) \)  we can define \( f = \lim_{s \to \infty}   \Gamma_{s+1}(\sigma)  \).  If we are acting to help some other \( \sigma' \) then \( \sigma' \) must appear to enter \( T' \) before \( \sigma \) finally enters \( T' \) \textemdash say at stage \( s_0 \) \textemdash  and  we only do so if \( \sigma' \) hasn't already met it's requirement.  Eventually all strings in \( T'_{s_0} \) that aren't permanently in \( T' \) will have left \( T' \) and the inductive assumption therefore tells us \( \Gamma(\sigma')\conv \) and once \( \Gamma_s(\sigma') \) reaches its limit we never act to help \( \sigma' \) again.  Thus, we can assume we are working above the finitely many extensions we make to help such \( \sigma' \) and that \( \recfnl{i}{f \Tplus X}{} \) is total since \textemdash if we aren't helping another string \textemdash every time we act we ensure \( \recfnl{i}{\Gamma_{s+1}(\sigma) \Tplus X}{} \supfunneq \recfnl{i}{\Gamma_{s}(\sigma) \Tplus X}{} \). 

We now argue that there is some string \( \xi \) so that the definition of \( \Gamma(\xi) \) will to help us. Specifically, we claim there is some string \( \xi \in T'  \) entering \( T' \) after \( \sigma \) such that 
\begin{equation}\label{eq:xi-helps-ri}
 \xi \incompat \sigma \land \recfnl{i}{\Gamma_s(\sigma) \Tplus X}{} \supfunneq \Gamma(\xi^{-})\conv \land \lh{\xi} > l \land \lh{\xi} \equiv 0 \pmod{4}
\end{equation}
We prove this by starting with \( \estr \) and always extending to be compatible with \( \recfnl{i}{f \Tplus X}{} \).  Now suppose that  for some \( \tau \in T' \) we have \( \Gamma(\tau)\conv \subfun \recfnl{i}{f \Tplus X}{}  \) then for some \( x \) we have  \( \Gamma(\tau)\concat[x] \subfun \recfnl{i}{f \Tplus X}{}  \).  If \( \tau\concat[x] \nin T' \) then for all sufficiently large \( s \) we have \( \Gamma_s(\tau\concat[x]) \supfun \Gamma(\xi)\concat[x]\concat[k] \) for \( k \) too large to be compatible with \( \recfnl{i}{f \Tplus X}{}  \) \textemdash and therefore  we would eventually stop having reasons to redefine \( \Gamma_s(\sigma) \) (since if \( \tau' \incompat \tau \) and \( \Gamma_s(\tau) = \Gamma(\tau) \) then \( \Gamma_s(\tau) \incompat  \recfnl{i}{f \Tplus X}{} \).  Thus, we can assume that  \( \tau\concat[x] \in T' \).   

But if \( \Gamma(\xi^{-})\conv\concat[x] \subfun \recfnl{i}{f \Tplus X}{}  \) and \( \xi = \xi^{-}\concat[x] \) and \( \xi \) enters \( T' \) before \( \sigma \) or \( \lh{\xi} \not\equiv 0 \pmod{4} \) or \( \lh{\xi} \leq l \) then, since \( \xi \in T' \), we will have \( \Gamma(\xi)\conv \) (either inductively or by above lemmas/discussion).  If \( \Gamma(\xi) \incompat \recfnl{i}{f \Tplus X}{}  \) then, as in the prior paragraph, we eventually stop extending \( \Gamma_s(\sigma) \) as we would if we had \( \xi = \sigma \) (since we only act when we see witnesses incompatible with \( \sigma \)).  Thus, by extending to keep our image under \( \Gamma \) compatible with \( \recfnl{i}{f \Tplus X}{}  \), we eventually find some \( \xi \in T' \) entering \( T' \) after \( \sigma \) satisfying \eqref{eq:xi-helps-ri}.    

Since \( \xi \in T' \) by  \cref{lem:almost-always-defined} for all sufficiently large stages \( \Gamma_s(\xi)\conv \).   Since \( \Gamma(\sigma')\conv \) for any \( \sigma' \) entering \( T' \) before \( \sigma \)  there is some finite stage larger than any mentioned so far in this proof after which  \( \Gamma_s(\xi) \) never again acts to help any string appearing to enter \( T' \) before \( \sigma \).  After this stage, if we act to ensure \( \recfnl{i}{\Gamma_{s+1}(\sigma) \Tplus X}{}  \) disagrees with something it is always  a string in \( T_s \)  extending \( \Gamma_s(\xi) \).   At the next time stage \( t \) at which  \( \xi \) receives attention after such an extension we will succeed in extending \( \Gamma_t(\xi) \) to capture the disagreement.  This will ensure that \( \Gamma_{t+1}(\xi) \incompat \recfnl{i}{\Gamma_{t+1}(\sigma) \Tplus X}{}  \).  Since \( \Gamma_{t'}(\xi) \supfun \Gamma_t(\xi) \) for \( t' \geq t \) this disagreement is permanently preserved and we will never again act to change our approximation to \( \Gamma(\sigma) \).

This contradicts the assumption that \( \Gamma(\sigma)\diverge \)  and completes the proof of \cref{lem:tprime-converge}.  To prove \cref{lem:req-ri-converges} assume that \( \lh{\sigma} = 4i > l \) and that \( f \in [T] \) is an extension of \( \Gamma(\sigma) \).  If there was \( g \in [T'] \) with \( \Gamma(g) = \recfnl{i}{f \Tplus X}{} \) then we would have some \( \xi \subfun g \) as above helping \( \sigma \) meet \req{R}{i} and since \( T \) is nice we would have had infinitely many chances to extend \( \Gamma(\xi) \) to allow  \( \recfnl{i}{\Gamma(\sigma) \Tplus X}{}  \) to disagree with it.  As we would have taken one of those chances this suffices to prove the claim.
\end{proof}

We now consider what happens when \( \sigma \nin T' \).

\begin{lemma}\label{lem:gamma-lim-dne}
If \( \sigma \nin T' \) then \( \Gamma(\sigma)\diverge \).  Moreover, if \( \recfnl{e}{\jump{X}}{\sigma}\conv =0 \) then \( \Gamma_s(\sigma)\diverge \) for almost all \( s \).   
\end{lemma}
\begin{proof}
Suppose \( \recfnl{e}{\jump{X}}{\sigma}\conv = 0 \).  For almost all stages the approximation to \( T'_s(\sigma) = 0 \).  By \cref{lem:inf-attention} there will be some stage \( s \)  after which \( \sigma \) never again enters \( T'_s \) and we set  \( \Gamma_s(\sigma)\diverge \).  As we will never again define \( \Gamma_{s'}(\sigma) \) for \( s' \geq s \) this verifies the moreover claim. 

Now if \( \sigma \nin T' \) then it will leave our approximation to \( T' \) infinitely often.  However, recall from our discussion above that while \( \Gamma_s(\sigma)\conv \) we only update \( T'_{s+1}(\sigma) \) when \( \sigma \) receives attention at stage \( s \).  Thus, if \( \sigma \nin T' \) there will be infinitely many stages at which we set  \( \Gamma_s(\sigma)\diverge \).  Even if we redefine \( \Gamma_s(\sigma) \) at the same stage so it never actually takes on an undefined value note that each time the second half of the stage defines \( \Gamma_s(\sigma) \) it chooses a different value than all prior values.  Thus the limit doesn't exist completing the proof.
\end{proof}

We verify that every path in \( T \) comes from a path in \( T' \).  In proving this we make use of the fact that we've clearly constructed \( \Gamma \) to be an expansionary function. 

\begin{lemma}\label{lem:paths-all-from-tprime}
If \( f \in [T] \) then there is some \( g \in [T'] \) with \( \Gamma(g) = f \).   
\end{lemma}
\begin{proof}
Suppose, for a contradiction, that \( f \in [T] \) but that \( f \) isn't the image of any \( g \in [T'] \).  As \( \Gamma \) is expansionary, it therefore follows that there must be some longest \( \sigma \) such that \( \Gamma(\sigma) \subfun f \).  Let \( x \) be such that \( \Gamma(\sigma)\concat[x] \subfun f \).  But once \( \Gamma(\sigma) \) settles and  \( \Gamma_s(\sigma\concat[x]) \) is guaranteed to never agree with \( f \) again \textemdash either because it has settled on it's own value or because it has become undefined and no string of the form \(  \Gamma(\sigma)\concat[x]\concat[k] \) with \( k \) chosen large after \( s \) will agree with \( f  \) \textemdash then we never again see any \( \tau \) with \( \Gamma_s(\tau) \supfun \Gamma(\sigma)\concat[x] \).  But this means we only add finitely many elements to \( T \) extending \( \Gamma(\sigma)\concat[x] \) contradicting the assumption that \( f \in [T] \).     
\end{proof}

We now verify that \( \Gamma \) is computable in \( \jump{X} \) provided \( T' \) was computable from \( X \).

\begin{lemma}\label{lem:gamma-comp-if-tprime}
If \( \recfnl{e}{\jump{X}}{} \) is total then \( \Gamma \Tleq \jump{X} \).   
\end{lemma}  
\begin{proof}
We note that \( \Gamma_s(\sigma) \) is a computable function of \( s \) and \( \sigma \) not just a partial computable function, i.e., the set \( \set{(s, \sigma)}{\Gamma_s(\sigma)\conv} \) is computable.  This follows directly from the rule for generating \( \Gamma_{s+1} \) from \( \Gamma_s \).  Since  \( \recfnl{e}{\jump{X}}{} \) by \cref{lem:tprime-converge,lem:gamma-lim-dne} we know that \( \lim_{s \to \infty} \Gamma_s(\sigma) \) always exists if we treat \( \diverge \) as if it was just a special symbol.  Since \( \Gamma_s(\sigma) \) is a computable function in the sense above it is  easy to see that \( \jump{X} \) can compute \( \Gamma(\sigma) \).    
\end{proof}

With these lemmas established, we now prove \cref{lem:tree-minus-one}.  

\begin{proof}
Evidently \( T \) is a tree, \cref{lem:gamma-comp-if-tprime} handles the computability of \( \Gamma \) and we already concluded that \( \Gamma \) is expansionary.  This leaves checking each of the numbered conditions.

(\ref{lem:tree-minus-one:comp}) The computability of \( T \) follows immediately from the fact that \( \sigma \in T \) iff \( \sigma \in T_{\godelnum{\sigma} + 1} \) since our stagewise approximations to \( T \) are uniformly computable in \( X \).     

(\ref{lem:tree-minus-one:nice}) \( T \) will be nice pursuant to the discussion in \cref{ssec:niceness}.

(\ref{lem:tree-minus-one:copy})  We directly stipulated that \( T\restr{l} = T'\restr{l} \) and that \( \Gamma \) is the identity on \( T'\restr{l} \).    

(\ref{lem:tree-minus-one:expansionary}) The construction only ever defines \( \Gamma(\sigma\concat[x]) \) to be some extension of \( \Gamma(\sigma)\concat[x] \).  

(\ref{lem:tree-minus-one:dom-zeta}) \Cref{lem:tprime-converge} proves \( \dom \Gamma \supset T' \) and \cref{lem:gamma-lim-dne} ensures \( \dom \Gamma \subset T' \).

(\ref{lem:tree-minus-one:homo}) This follows from \cref{lem:paths-all-from-tprime} via  \cref{lem:expan-homeo}.         

(\ref{lem:tree-minus-one:subg}) This is immediate from \cref{lem:req-gi-converges} as we meet the requirements \req{G}{i}. 

(\ref{lem:tree-minus-one:jump-inv}) By \cref{lem:req-gi-converges} if \( f \in [T] \)  we can compute \( \jump{(f \Tplus X)} \) from \( f \Tplus \jump{X} \) by searching for \( \sigma \subfun f \) that either witnesses \( i \in \jump{(f \Tplus X)} \) or witnesses the fact that no extension of \( \sigma \) in \( T \) has that property.   Given \( f \in [T] \) we can compute \( g \in [T'] \) \textemdash guaranteed to exist by \cref{lem:paths-all-from-tprime} \textemdash with \( \Gamma(g) = f \) uniformly in \( \jump{X} \) since even if \( \dom \Gamma \) isn't computable in \( \jump{X} \) it is r.e. in \( \jump{X} \).   Thus, \( f \Tplus \jump{X} \) can enumerate \( \sigma \in T' \) with \( \Gamma(\sigma) \subfun f \) and thereby compute \( g \).  Computing \( f \) from \( g \Tplus \jump{X} \) follows by similar reasoning.  

  (\ref{lem:tree-minus-one:path-noncomp}) This is immediate from the fact that \( T \) is nice and \cref{lem:req-gi-converges} since niceness ensures that if \( \recfnl{i}{X}{} \) is total then no path through \( T \) can strongly avoid the set of strings which disagree with \( \recfnl{i}{X}{} \).

  (\ref{lem:tree-minus-one:pair-noncomp}) This is immediate from \cref{lem:req-ri-converges}.

  (\ref{lem:tree-minus-one:splitting}) This is immediate from \cref{lem:req-si-converges}. 
\end{proof}

\subsection{Building The Tower}\label{ssec:build-std-tower}

We now sketch a proof of \cref{thm:standard} \textemdash the standard version of the main theorem.  We only sketch a proof here as we will later give a full proof of this result based on the construction we use in the next section to prove the non-standard version of the theorem.  However, it may be helpful to understand how the proof works in the standard case before tackling the more complicated machinery in the next section.  Readers who want more detail about how the standard construction would work are urged to consult \cite{Gerdes2023PiSingleton} (particularly the appendix) \textemdash where a similar result was proved for standard ordinals but with a different version of the pulldown lemma that worked in double jumps..

To prove \cref{thm:standard} it suffices to build a subgeneric tower of height \( \alpha \) with \( T_{\alpha} = T \) where \( T \) is some \( \zeron{\alpha} \) computable tree.  
 We face a few challenges to constructing this subgeneric tower of height \( \alpha \).  First, we need to extend the idea that \( T_n\restr{n} = T_\omega\restr{n} \) into the transfinite.  The obvious generalization would be to say that \( T_{\Olim{\lambda}{n}}\restr{n} = T_{\lambda}\restr{n} \).  But what if \( \Olim{\lambda}{n} \Oadd 1 \) is actually \( \Olim{\alpha}{m} \) for some other notation \( \alpha \)?  How do we make sure all the demands imposed by higher levels are compatible and automatically satisfied when we use \hyperref[lem:tree-minus-one]{the pulldown lemma} to build \( T_{\beta} \) from \( T_{\beta \Oadd 1} \)?  One answer is to build a computable function of \( \alpha, \beta \) that tells us the length of strings \( T_{\alpha} \) and \( T_\beta \) should agree on \textemdash at least for strings which extend to infinite paths \textemdash which has the following properties.

    \begin{definition}\label{def:good-copylen}
Say that \( l^{\alpha}_{\gamma} \) is an acceptable copy-length function for a tower \(T_{\beta}, \Gamma, I \) if \( l^{\alpha}_{\gamma} \) is a computable function of \( \alpha, \gamma \) that is defined and satisfies the following whenever \( \gamma \Oless \alpha \in I \) 
\begin{enumerate}
\item\label{def:good-copylen:limit} \(\displaystyle \lim_{n \to \infty} l^{\alpha}_{\Olim{\alpha}{n}} = \infty \text{ when } \alpha \text{ is a limit notation} \)
\item\label{def:good-copylen:between}  \( \displaystyle l^{\alpha'}_{\gamma'} \geq l^{\alpha}_\gamma \) whenever \( \displaystyle \gamma \Oleq \gamma' \Oless \alpha' \Oleq \alpha  \) 
\end{enumerate} 
\end{definition}

   Another problem is that merely specifying how far we should agree with various other trees isn't quite enough since we want a specific definition of \( T_{\beta} \) in terms of copying from one other tree\footnote{While it's not strictly necessary to identify a single source tree to copy from, it is necessary to ensure that larger limit notations only determine \( T_{\beta}\restr{l} \) for some finite bound \( l \) and to turn this into an effective construction of \( T_{\beta} \).  Thus, while \cite{hjorth_argument_2003} doesn't explicitly identify a particular notation to copy from Hjorth still grapples with the same issue in his definition of \( \alpha_{n, \lambda} \).}, say, \( T_{\beta^{\Diamond}} \) where \( \beta^{\Diamond} \) is some limit notation between \( \beta \) and \( \alpha \).  This has to be chosen so that demanding \( T_{\beta} \) copy \( T_{\beta^{\Diamond}} \) on strings of length at most \( l^{\beta^{\Diamond}}_{\beta} \) is enough to guarantee that our acceptable copy-length function is respected.   A uniform method for identifying such a \( \beta^{\Diamond} \)  is given in  \cite{Gerdes2023PiSingleton} and in the next section a similar special notation will fall out of our tree representation.

   Our final problem is that if we want \( T_{\beta} \) to copy some initial segment of \( T_{\beta^{\Diamond}} \)  we need to be able to ensure that the part copied down to \( T_\beta \) is \( \zeron{\beta} \) computable.  We reproduce a version of the lemma from \cite{Gerdes2023PiSingleton} that, when \( \lambda \) is a limit, will allow us to find a version of \( T_{\lambda} \),  \( T^{*}_{\lambda} \),  where membership in \( T_{\lambda}\restr{l_n} \) is uniformly computable in \( \zeron{\Olim{\lambda}{n}} \).    

  \begin{lemma}\label{lem:tree-uniform-limit}
  If \( T \Tleq X \), \( X \Tequiv \TPlus_{i \in \omega} X_i \) and \( l_n \) is a strictly monotonic computable function of \( n \) then there is a tree \( T^{*} \supset T \) with \( [T^{*}] = [T] \) such that  \( T^{*}\restr{l_n} \) is uniformly computable in \( X_{\leq n} = \TPlus_{i \leq n} X_i \).   Furthermore, if we already have \( T\restr{l_{-1}} \) is computable in \( X_0 \) we can set \( T\restr{l_0} = T^{*}\restr{l_0} \). This holds with full uniformity and if we know that \( T \) is nice we can ensure that  \( T^{*} \) is nice as well.   
  \end{lemma}
  Here, full uniformity means that indexes for \( T^{*}\restr{l_n} \) as \( X_{\leq n} \) computable sets can be computed from indexes for \( T \) and the function \( l_n \) for the main result.  For the furthermore claim, one also needs \( l_{-1} \) and an index for \( T\restr{l_{-1}} \) as a \( X_0 \) computable set. 
  \begin{proof}
  WLOG we can assume that \( X = \TPlus_{i \in \omega} X_i \) and that \( \recfnl{e}{X}{\sigma} \) is the characteristic function of \( T \).  Now let \( u_n \) be the largest value such that every \( x < u_n \) is of the form \( \pair{n'}{y} \) with \( n' \leq n \), i.e., the largest use of a computation from \( X \) which only consults \( X_{\leq n} \).  Given \( \sigma \) let \( n \) be the least value such that \( l_n \geq \lh{\sigma} \) and place \( \sigma \in T^{*} \) iff there is no \( l \leq \lh{\sigma} \) and \( s \leq u_n \) such that \( \recfnl[s]{e}{X}{\sigma\restr{l}}\conv = 0 \).

  Thus, \( T^{*}\restr{l_n} \) consists of those strings such that no substring is computed to be out of \( T \) via a computation that only consults \( X_{\leq n} \).  Thus, \( T \subset T^{*} \) and  \(  T^{*}\restr{l_n} \Tleq X_{\leq n} \).  The uniformity is obvious and given any \( f \nin [T] \) there is some \( \sigma \subfun f \) not in \( T \).  Thus, all sufficiently long extensions of \( \sigma \) aren't in \( T^{*} \) so \( [T^{*}] = [T] \).   The furthermore is a trivial modification and the uniformity is clear.  Since \( T^{*} \supset T \) and \( T \) is nice it is enough to keep \( \sigma\concat[i] \) out of \( T^{*} \) whenever \( \lh{\sigma} \equiv 2 \pmod{4} \) and either \( i > 1 \) or \( \sigma\concat[1-i] \nin T^{*} \).               
  \end{proof}  

  With this version of the lemma in hand we sketch the approach to building a subgeneric tower.  As a tower of height \( \alpha \)  consists of two computable functionals \textemdash one which builds \( T_\beta \) from \( \zeron{\beta} \) and  \(  \Gamma^{\gamma}_{\beta}(\zeron{\beta}; \sigma)  \)  we can think of a tower as coded by an r.e. set with some index \( i \).  We now define a computable function \( q(i) \) so that \( q(i) \) is an index for a tower whose top element is \( T \) \textemdash or at least a version modified to be nice \textemdash and which defines \( T_{\beta} \) by using \hyperref[lem:tree-minus-one]{the pulldown lemma} to build it from the version of \( T_{\beta \Oadd 1} \) given by \( i \).  We will also, as described above, computably identify some limit notation \( \lambda \in I \) above \( \beta \) and would ideally ensure that \( T_{\beta}\restr{l^{\lambda}_{\beta}} =   T_{\lambda}\restr{l^{\lambda}_{\beta}} \) regardless of how \( i \) tells us  \( T_{\beta \Oadd 1} \) behaves.  However, we must slightly relax this demand to merely requiring equality on strings which extend to paths and otherwise allow \( T_{\beta}\restr{l^{\lambda}_{\beta}} \supset   T_{\lambda}\restr{l^{\lambda}_{\beta}} \) so we can use \cref{lem:tree-uniform-limit} to ensure the parts copied down are appropriately computable.  We define \( \Gamma^{\beta \Oadd 1}_{\beta} \) to be the functional given by \hyperref[lem:tree-minus-one]{the pulldown lemma} and insist that \( \Gamma^{\lambda}_{\beta}(\sigma) \) should be equal to something like \( \Gamma^{\Olim{\lambda}{\lh{\sigma}}}_{\beta}(\sigma) \).   Finally, we use the recursion theorem to find a fixed-point \( i \)  for \( q \) and we claim that \( i \) codes for the desired tower.      

    This amounts to saying that we use \hyperref[lem:tree-minus-one]{the pulldown lemma} to build trees from their successors and at limit stages we ensure that as \( n \) goes to infinity \( T_{\Olim{\lambda}{n}} \) agrees with \( T_{\lambda} \) on longer and longer initial segments on which \( \Gamma^{\lambda}_{\Olim{\lambda}{n}} \) is the identity.  While this captures the intuition, as we will see in the next section there are many annoying details which need to be managed by the construction to give a rigorous proof.

\section{Non-Standard Construction}\label{sec:non-standard-tower}

To prove \cref{thm:non-standard} we need to somehow modify the approach above to build a subgeneric tower of height \( \wck \), i.e., a subgeneric tower which has a tree for every notation in some path \( \kleeneO1 \) through \( \kleeneO \).   One approach to this problem is to reason model-theoretically and argue there should be  non-standard models where \cref{thm:standard} holds and apply that result to a non-standard ordinal notation in such a model.  In this context, non-standard means believing that certain linear orders are actually well-orders even though they aren't truly well-ordered.  Intuitively, if we applied this theorem to these non-standard notations we could extract a computable tree whose paths are \( < \wck \) generic since the non-standard notation would bound all the standard notations.

Unfortunately, model theory doesn't offer us a magic wand we can wave over our proof in the standard case to produce the non-standard result.  Making the above argument rigorous would require careful work to ensure that not only do the desired non-standard models exist but also to establish the appropriate absoluteness results to show that the computable tree \( T_0 \) we extract has the appropriate properties.  This is a perfectly reasonable way to approach the problem, but we instead choose to give a more concrete version of the proof framed in a computability theoretic context.  

  The model-theoretic intuition that there are in some sense `fake' ordinal notations above the true notations will still guide our approach.  However, instead of working with some model's idea of a non-standard ordinal notation we will work with the top element of a particular computable linear order that isn't a well-order but doesn't have any hyperarithmetic decreasing sequences.   This has the advantage of letting us leverage the approach we used in \cref{ssec:prob71} of turning a computable tree with only non-hyperarithmetic branches into a representation of a path through \( \kleeneO \). 

   \subsection{Construction Overview}

  We will start with a tree \( U \) which has only non-hyperarithmetic branches, define \( U^{+}, \rho \) as in \cref{ssec:prob71} and let \( f_U \) be the leftmost path through \( U \). Thus,  \( U^{+} \) is also a computable tree with \( [U^{+}] = [U] \) and we have a computable function \( \rho \) mapping elements in \( U^{+} \) to the ordinal notations for their height under \( \KBless \).  Our subgeneric tower will be defined with \( I \eqdef \kleeneO1 \eqdef \set{\rho{\sigma}}{\sigma \in U^{+} \land \sigma \leftof f_U}  \) where we recall that \( U^{+}_{\leftof f_U} \) is the well-founded initial segment of \( U^{+} \) under \( \KBless \) and therefore \( \kleeneO1 \) is a path through \( \kleeneO \).   We will start with \( T_{\estr} \) (the top tree in our `tower') and from this define \( T_{\sigma} \) for all \( \sigma \in U^{+} \).  The resulting definition will appear to be subgeneric tower defined with \( I \) equal to the image of \( U^{+} \) under \( \rho \)  but technically only the well-founded initial segment will qualify as a subgeneric tower.    

  The reason for using the tree \( U^{+} \) rather than just working abstractly with respect to some unique set of notations \( \kleeneO1 \) is similar to the reason we used such a tree in \cref{ssec:prob71}  \textemdash having elements above all standard notations helps us deal with all notations in \( \kleeneO1 \) in a uniform and coherent fashion.  In particular, the presence of a concrete tree will give us a much more elegant way to solve the problem from \cref{ssec:build-std-tower} of deciding what tree to copy from than the approach in \cite{Gerdes2023PiSingleton}.

For notational simplicity we will often write \( T_{\vartheta} \) instead of \( T_{\rho(\vartheta)} \) and \( \zeron{\vartheta} \) instead of \( \zeron{\rho(\vartheta)} \) but will use \( \rho \) as necessary to avoid confusion.  Since \( \rho \) is monotonic on \( U \) we can elide the distinction elements of \( U^{+} \) and their images without concern even in constructions, e.g., we can define \( \zeron{\vartheta} \) in terms of columns \( \sigma \KBless \vartheta \) rather than \( \rho(\sigma) \) and be assured the result will be equivalent up to Turing degree.  Also, as we did in the last sentence, it will be convenient to assume the relation  \( \KBless \) is restricted to elements in \( U^{+} \) so that \( \theta \KBless \vartheta \) entails both strings are in \( U^{+} \).

  We build a tower and a computable function \( l^{\sigma}_{\tau} \) satisfying the criteria of an acceptable copy-length function \textemdash though formally it is \( l^{\rho(\sigma)}_{\rho(\tau)} \) that satisfies  \cref{def:good-copylen} \textemdash which the tower will respect.  In addition to the computable function \( l^{\sigma}_\tau \) this means building two computable functionals, \( \Upsilon, \Gamma \) with \( \Upsilon \) responsible for computing \( T_{\sigma}  \)  from  \(\zeron{\sigma} \).   While the surface form of the construction below will be a direct construction of \( \Gamma \) and the trees \( T_{\sigma} \) keep in mind that the true construction will be of a computable function \( q \) transforming an index for the pair \( \Upsilon, \Gamma \) and the true functionals \( \Upsilon, \Gamma \) will be those whose index is given by a fixed point of \( q \).   


  \subsection{Defining \texorpdfstring{\( \zeron{\theta} \)}{\textbf{0}ᶿ}}\label{ssec:defining-zero-theta}

  Since we need to define \( T_{\theta} \)  even when \( \theta \nin U^{+}_{\leftof f_U} \), e.g., \( T_{\estr} \), this raises the question of how to understand \( \zeron{\theta} \) when \( \rho(\theta) \nin \kleeneO \).  While it might seem like it should be enough to know that any nice\footnote{For instance, that every element is a finite successor of \( 0 \) or a limit.} computable linear order doesn't admit any infinite  descending hyperarithmetic sequences  to define something which behaves like a jump hierarchy on it this is refuted in 
 \cite{Friedman1976Uniformly}.  However, that same paper motivates the following lemma which solves the problem.

  \begin{lemma}\label{lem:unif-no-hyp-and-jump}
    There is a computable tree \( U \) with \( [U] \neq \eset \) but \( [U] \isect \HYP = \eset \) such that \( U^{+} \) under \( \KBless \) supports a jump hierarchy.   This result relativizes with all possible uniformity. 
  \end{lemma}

  Intuitively, supporting a jump hierarchy just means allowing for the definition of \( \zeron{\sigma} \) in a way that satisfies all the expected inductive properties.  Formally speaking, it is understood to mean that there is a function \( h \) satisfying  \eqref{eq:def-H-recursive}.

  \begin{proof} 

  Let \( A(h, e) \) be a  predicate asserting that \( e \) is an r.e. index for a tree \( U \) which satisfies the following inductive definition for all strings \( \theta \).  Recall that in  \cref{ssec:prob71} we defined \( U^{+}, \rho \) in terms of \( U \) as well as the operations \( \vartheta^{+} \) \textemdash giving the successor of \( \vartheta  \) under \( \KBless \) \textemdash  and \( \eta(\theta) \) \textemdash giving the \( \KBless \) least element \( \tau \supfun \theta \) (so \( \eta(\estr) \) should be the \( 0 \) element).

\begin{equation}\label{eq:def-H-recursive}
     \setcol{h}{\theta} = \begin{cases}
                        \eset & \text{if } \theta \nin U^{+} \\
                        X & \text{if } \theta = \eta(\estr) \\
                        \TPlus_{n \in \omega} \setcol{h}{\theta\concat[n]} & \text{if } \theta \in U\\
                        \jump{(\setcol{h}{\vartheta})}  & \text{o.w. where } \theta = \vartheta^{+} \\
                         \end{cases}
  \end{equation}  

    We also implicitly assume that \( A(h,e) \) only holds if \( \eta(\estr)\conv \), i.e., \( U^{+} \) has a \( \KBless \) least element, and that \( \vartheta^{+} \) is defined for every \( \vartheta \in U^{+} \). 

 We claim that \( A(h,e) \) is (up to Turing degree) a \( \pizn[\zeroj]{1} \) property and that we can define \( \hat{A}(h_0 \Tplus h, e) \) where this is a \( \pizn{1} \)  property which holds only if \( h_0  = \zeroj \) and \( A(h, e) \) holds.  We don't give details about our specific definition of \( \hat{A} \) since it would have sufficed to simply use the fact that \( A(h,e) \) was \( \sigmain{1} \).  This gives us a computable function \( f(e) \) such that \( f(e) \) codes for the computable tree given by \( \hat{A}(\cdot, e) \), i.e., a tree whose paths are the solutions to \( \hat{A}(\cdot, e) \).      

  Let \( U \) be the computable tree whose index is a fixed point of \( f \).  Since \( \hat{A}(\cdot, e) \) is computable regardless of the value of \( e \) we end up with a computable, not merely r.e., tree \( U \).  We now argue that this tree has the desired property.  If \( [U] = \eset \) then \( \KBless \) is a well-order so there is clearly some solution \( h \) to \( A(h,e) \) and therefore  \( \zeroj \Tplus h \) is a path through the tree given by \( \hat{A}(\cdot, e) \).  But this tree is also \( U \) contradicting the assumption.  Thus, \( [U] \neq \eset \).

  However, any path through \( U \) computes the function \( h \) satisfying \cref{eq:def-H-recursive} relative to \( U \) thereby witnessing that \( \KBless \) applied to \( U^{+} \) supports some jump hierarchy.   This entails that \( U \) can't have any hyperarithmetic branches since there is no ill-founded hyperarithmetic jump hierarchy (see lemma 3.4 in Ch 3 of \cite{Sacks1990Higher}).  Proving the relativized version simply requires dragging the variable \( X \) through the entire proof and considering the tree of solutions to \( \hat{A}(X,h,e) \) whose first component is \( X \).  
  \end{proof}   

  To keep the presentation readable, we will conduct the rest of our construction as if we have sets \( \zeron{\theta} \) defined for every \( \theta \in U^{+} \) satisfying the recursive definition given above.  That is, we assume that  \( \zeron{\theta} \eqdef  \setcol{h}{\theta}  \)  where \( h \) satisfies \cref{eq:def-H-recursive} for all strings.  However, while \( \zeron{\theta} \) will have a unique definition when \( \theta \in U^{+}_{\leftof f_U} \) the collection of paths through \( U \) will necessarily be perfect with each path corresponding to a way to assign a jump hierarchy to the elements of \( U^{+} \).  Ultimately, this choice won't matter since we only care about the computable tree \( T_0 \) and properties about its paths that are determined on the well-founded initial segment of \( U^{+} \).  We merely need to know that there is some way to assign the sets \( \zeron{\theta} \) to satisfy \cref{eq:def-H-recursive} to verify that \( [T_0] \) contains a subset homeomorphic to \( \baire \).  

  This is only one way to handle this problem.  Alternatively, we could have defined a tower of subtrees of \( \wstrs \cross \wstrs \) with paths through the first component meant to represent the \( \zeron{\beta} \) sets.  This is the approach taken by Hjorth in his manuscript  \cite{hjorth_argument_2003}  solving problem 57\( ^* \).  

  Before we move on, we prove a utility lemma we will need later.

 \begin{lemma}\label{lem:lim-of-mins}
 If \( \sigma \in U \) then \( \zeron{\sigma} \Tequiv \TPlus_{n \in \omega} \zeron{\eta(\sigma\concat[n])} \) and this holds uniformly and is independent of the choice of jump hierarchy. 
 \end{lemma}
 Recall that \( \eta(\sigma\concat[n]) \) is giving the least element extending \( \sigma\concat[n] \).  By independent of the choice of jump hierarchy we mean that the computation witnessing the equivalence works for any assignment satisfying \cref{eq:def-H-recursive}. 
 \begin{proof}
 It is clear that the claim would be true if the right hand side was replaced with  \( \TPlus_{n \in \omega} \zeron{\sigma\concat[n]}  \).  Thus, it is enough to observe that the successor of \( \sigma\concat[n] \) is always \( \eta(\sigma\concat[n+1]) \) when \( \sigma \in U \) and noting that this gives us a uniform computation of  \( \zeron{\sigma\concat[n]} \) from \( \zeron{\eta(\sigma\concat[n+1])} \) that works for any jump hierarchy defined from a path satisfying \cref{eq:def-H-recursive}. 
 \end{proof}  

 Note that, we can assume that whenever we produce an index for a \( \zeron{\sigma} \) computable set \( X \)  we can assume that index is also an index for \( X \) as a \( \zeron{\tau} \) computable set for any \( \tau \KBgtr \sigma \).   This is trivial given our definition of \( \zeron{\sigma} \) as we merely need to ignore the columns in \( \zeron{\tau} \) that aren't equal to (the code of) some \( \sigma' \KBleq \sigma \).

\subsection{How Much To Copy?}

We make the following inductive definition with the idea that \( l(\sigma) \) indicates how much of \( T_{\sigma^{-}} \) the tree \( T_{\sigma} \) should copy.   Note how the use of the concrete \( \KBless \) ordering on \( U^{+} \) gives us an elegant answer to the question raised in \cref{ssec:build-std-tower} of which tree to copy from.

\begin{equation}\label{eq:l-of-sigma}
l(\sigma) \eqdef \begin{cases}
                \diverge & \text{if } \sigma \nin U^{+}\\
                0 & \text{if } \sigma = \estr\\
                l(\tau) + 4(n+1) & \text{if } \sigma = \tau\concat[n]\\
                \end{cases}
\end{equation}  
This definition guarantees that for all \( \sigma \),  \( l(\sigma) \equiv 0 \pmod{4} \).  We use this function to define a computable copy-length function \( l^{\sigma}_\tau \). 

\begin{equation}\label{eq:l-of-sigma-tau}
l^{\sigma}_\tau  \eqdef \begin{cases}
                                                            \diverge & \text{unless } \sigma, \tau \in U^{+} \text{ and }  \tau \KBless \sigma \\
                                                            l^{\sigma}_{\tau^{-}} & \text{if } \tau^{-} \nsubfun \sigma \\
                                                            l^{\sigma^{-}}_{\tau} & \text{o.w. if } \sigma^{-} \nsubfun \tau \\
                                                             l(\tau) & \text{if } \tau^{-} = \sigma^{-} \text{ or } \tau^{-} = \sigma  \\
                                                            \end{cases}
\end{equation}

\begin{lemma}\label{lem:l-works}
The computable function \(  l^{\sigma}_\tau \) is defined whenever \( \tau \KBless \sigma  \) and is an acceptable copy-length function.  Moreover, we always have \( l^{\sigma}_\tau = l(\nu) \) where \( \tau \supfun \nu  \) and \(  \nu^{-} = \sigma \meet \tau \).  
\end{lemma}
Recall that \( \sigma \meet \tau \) denotes the longest common initial segment of \( \sigma \) and \( \tau \).   Formally, what actually satisfies \cref{def:good-copylen} is the function \( l^{\rho(\sigma)}_{\rho(\tau)} \) but we continue to elide this distinction. 
\begin{proof}
First, suppose that \( l^{\sigma}_\tau \) isn't defined where \( \sigma, \tau  \in U^{+} \) are a lexicographically minimal pair of strings  relative to string extension (not \( \KBless \))  \( l^{\sigma}_\tau \) isn't defined.  If either \( \sigma^{-} \nsubfun \tau \) or \( \tau^{-} \nsubfun \sigma \) then we would clearly have \( l^{\sigma}_\tau \) defined.  But if \( \tau^{-} \subfun \sigma \) and \( \sigma^{-} \subfun \tau \)  we have \( \tau^{-} = \sigma^{-} \) or \( \tau^{-} = \sigma \)  which also implies \( l^{\sigma}_\tau\conv \).  To verify the moreover claim note that if we start with \(\nu = \tau \) and continue to shorten \( \nu \) until \( \nu^{-} \subfun \sigma \) we eventually end up with \( \nu^{-} = \sigma \meet \tau \) and that we can shorten \( \sigma \) until either \( \sigma \) or \( \sigma^{-} \) is equal to \( \sigma \meet \tau \) without changing \( l^{\sigma}_{\nu} \).

It is easy to see that when \( \sigma \in U \), i.e., \( \rho(\sigma) \) is a limit, then  \( l^{\sigma\concat[n]}_\sigma \) is defined and a strictly monotonically increasing function in \( n \) with limit \( \infty \).  This establishes part \ref{def:good-copylen:limit} of \cref{def:good-copylen}.  To verify part \ref{def:good-copylen:between} assume that \( \tau \KBleq \tau' \KBless \sigma' \KBleq \sigma \) and let \( \xi = \tau \meet \sigma \) and \( \xi' = \tau' \meet \sigma' \).  Observe that \( \xi \subfun \xi' \) and that if \( \xi\concat[x] \subfun \tau \) then either \( \xi' \)  extends \( \xi\concat[y] \) for some \( y \geq x \) or \( \xi = \xi' \) and \( \tau' \) extends \( \xi\concat[y] \) for some \( y \geq x \).  By the moreover remark either way vindicates the fact that \( l^{\sigma}_\tau \geq l^{\sigma'}_{\tau'} \). 
\end{proof}

 \subsection{A Tower of Trees on a Tree}\label{ssec:nonstandard-construction}

 We now, finally, give a rule for building \( T_{\vartheta} \).  If \( \vartheta = \estr \) then \( T_{\vartheta} = \wstrs \).  Otherwise, we build \( T_{\vartheta} \) by first copying \( T_{\vartheta^{-}} \) (see below) \textemdash or \( \wstrs \) if \( \vartheta^{-} = \estr \)  \textemdash  on strings of length \( l(\vartheta) \) and then use the pulldown lemma to define the rest of \( T_{\vartheta} \) from \( T^{\dagger}_{\vartheta^{+}} \) \textemdash here we are using superscript \( \dagger \) to indicate the version of an object provided by the recursion theorem.  Recall that the overall form of the construction is to define a computable function \( q(i) \) which produces an index for the computable functional \( \Upsilon \) where \( \Upsilon(\zeron{\vartheta}) = T_{\vartheta} \).  So \( T^{\dagger}_{\vartheta^{+}}  \) is just a more expressive name for \( \Phi_i(\zeron{\vartheta^{+}})  \).     

 Unfortunately, this isn't quite sufficient because \( T_{\vartheta^{-}}\restr{l(\vartheta)} \) may not be computable in \( \zeron{\vartheta} \), much less uniformly computable.  We could solve this problem using \cref{lem:tree-uniform-limit} to generate  \( T^{*}_{\vartheta^{-}} \) \textemdash a version of \( T_{\vartheta^{-}} \) such that  \( T^{*}_{\vartheta^{-}}\restr{l(\vartheta)} \Tleq \zeron{\vartheta} \).  However, it will be convenient to ensure that \( T_{\vartheta^{+}}\restr{l(\vartheta)} \) really is equal to  \( T^{*}_{\vartheta^{-}}\restr{l(\vartheta)} \) \textemdash the initial segment that will be copied by \hyperref[lem:tree-minus-one]{the pulldown lemma} in constructing \( T_{\vartheta} \).  To guarantee this nice feature, it is enough to ensure that when we construct any \( T^{*}_{\theta} \) where \( \vartheta^{-} \subfunneq \theta \subfunneq \vartheta^{+} \) we keep \( T^{*}_{\theta}\restr{l(\vartheta)} = T_{\theta}\restr{l(\vartheta)} \).      

 To this end, given \( T_{\vartheta} \) \textemdash except if \( \vartheta = \estr \) in which case \( T^{*}_{\estr} = T_{\estr} = \wstrs \) \textemdash  we will define \( T^{*}_{\vartheta} \) so that \( T_{\vartheta}\restr{l(\vartheta)} = T^{*}_{\vartheta}\restr{l(\vartheta)} \).  Since we will copy \( T_{\vartheta}\restr{l(\vartheta)}  \) down to any \( T_{\theta}, \theta \supfun \vartheta \) we need to ensure that \( T_{\vartheta}\restr{l(\vartheta)} \) is already computable in \( \zeron{\theta} \) for any \( \theta \supfun \vartheta \).  We do this by using \cref{lem:lim-of-mins} in conjunction with \cref{lem:tree-uniform-limit} so that \( T^{*}_{\vartheta}\restr{l(\vartheta\concat[n])} \) is always computable in \( \zeron{\eta(\vartheta\concat[n])} \) \textemdash recall that \( \eta(\vartheta\concat[n]) \) is the \( \KBless \) minimal element extending (or equal to) \( \vartheta\concat[n] \).  This allows us to use the moreover part of \cref{lem:tree-uniform-limit} to leave  \( T_{\vartheta}\restr{l(\vartheta)} = T^{*}_{\vartheta}\restr{l(\vartheta)}  \) since that part of \( T_{\vartheta} \) was already computable in \( \zeron{\eta(\vartheta)} \) and therefore \( \zeron{\eta(\vartheta\concat[0])} \).

 Importantly, the only place that we use our guess \( i \) at an index for \( \Upsilon \) is to provide the indexes for trees of the form \( T_{\vartheta^{+}} \) for use in \hyperref[lem:tree-minus-one]{the pulldown lemma}.  We \textit{don't} define \( T_{\vartheta} \) by directly copying from \( T^{\dagger}_{\vartheta^{-}} \) but by copying from the version of \( T_{\vartheta^{-}} \) the construction produces using \hyperref[lem:tree-minus-one]{the pulldown lemma}.  This ensures that all uses of our guess \( i \) in the construction are wrapped inside an application of \hyperref[lem:tree-minus-one]{the pulldown lemma} and therefore can take advantage of the aspects of that lemma which ensure we fail gracefully and thereby avoid trivial solutions\footnote{If we'd simply defined \( T_{\vartheta} \) to copy directly from \( T^{\dagger}_{\vartheta^{-}} \) then a valid fixed point of \( q \) might result in \( T_{\vartheta} \) always being empty. }.

 Formally speaking, we can define a computable function \( f(\vartheta, i) \) which yields an index for \( T_{\vartheta} \) as a \( \zeron{\vartheta} \) computable set\footnote{That is, an index for the characteristic function of \( T_{{\vartheta}} \) with computability only guaranteed if that function is total.}.    We define \( f(\vartheta, i) \) to be an index for \( \wstrs \) if \( \vartheta = \estr \) and if \( \vartheta \neq \estr \)  the result of applying  \hyperref[lem:tree-minus-one]{the pulldown lemma}.  In the later case, we need to specify \( l \),  \( e_0 \) \textemdash an index for \( T'\restr{l} \) as a \( \zeron{\vartheta} \) computable set \textemdash  and \( e \) \textemdash an index for \( T' \) as a \( \zeron{\vartheta^{+}} \) computable set.

 We take \( e_0 \) to be an index for \( T^{*}_{\vartheta^{-}}\restr{l(\vartheta)} \) where \( T^{*}_{\vartheta^{-}} \)  is the result of applying \cref{lem:tree-uniform-limit} (in light of \cref{lem:lim-of-mins}) to \( T_{\vartheta^{-}} \) as described above \textemdash or really to the index \( f(\vartheta^{-}, i) \) for that tree \textemdash  so that \( T^{*}_{\vartheta^{-}}\restr{l(\vartheta)} \) is uniformly computable in \( \zeron{\eta(\vartheta)} \).  Note that there is a single definition of \( T^{*}_{\vartheta^{-}} \) used in defining \( T_{\vartheta} \) for all immediate extensions \( \vartheta \) or \( \vartheta^{-} \).  We let \( e \) be the index for \( T^{\dagger}_{\vartheta^{+}} \), i.e., the guess at \( T_{\vartheta^{+}} \) inferred from the index \( i \), and let \( l = l(\vartheta) \).     A trivial induction on the length of \( \vartheta \) will verify that \( f(\vartheta, i) \) is defined for every \( \vartheta \in U^{+} \). 

 We then take \( f \) and use it to define a computable function \( q(i) \) which gives an index for the computable functional \( \Upsilon \) defined by \( \Upsilon(\zeron{\vartheta}) = \recfnl{f(\vartheta, i)}{\zeron{\vartheta}}{} \), i.e., \( \Upsilon \)  uses the index \( f(\vartheta, i) \) to compute \( T_{\vartheta} \) from \( \zeron{\vartheta} \).  By the recursion theorem we can then find a fixed point \( i \)  for \( q \) which defines a computable functional \( \Upsilon \) giving us the final definition of \( T_{\vartheta} = \Upsilon(\zeron{\vartheta})  \).

 While this suffices to define our trees, we also need to define our expansionary functions \( \Gamma^{\vartheta}_\theta \) which we do now.  We also use the recursion theorem to define \( \Gamma^{\vartheta}_\theta \) but we can assume that we've already succeeded in defining the trees \( T_{\vartheta} \) for all \( \vartheta \in U^{+} \) and thus that \( \Gamma^{\vartheta^{+}}_{\vartheta} \) is defined for us for all \( \vartheta \in U^{+} \) by the application of \hyperref[lem:tree-minus-one]{the pulldown lemma}.  We can then define \( \Gamma \) in all other cases as follows.  Here we use \( \murec{n}{\Psi(n)} \) to indicate the least \( n \) satisfying \( \Psi(n) \) and again use \( \dagger \) to indicate the version of \( \Gamma \) given by our guess at the index \( i \).

 \begin{equation}\label{eq:gamma-rec-def}
 \Gamma^{\vartheta}_{\theta}(\sigma) \eqdef \begin{cases}
                                        \diverge & \text{if } \vartheta \nin U^{+} \lor \theta \nin U^{+} \lor \vartheta \nKBless \theta \lor \sigma \nin T_{\vartheta} \\[.5ex]
                                        \sigma & \text{if } \vartheta = \theta \\[.5ex]
                                        \bigl(\Gamma^{\dagger, \epsilon}_{\theta}\!\!\!\circ\Gamma^{\epsilon^{+}}_{\epsilon}\bigr)(\sigma) & \text{if } \vartheta = \epsilon^{+}\\[.5ex]
                                        \Gamma^{\dagger, \vartheta\concat[m]}_{\theta}(\sigma) & \text{if } \vartheta \in U \land m = \murec{n}{\theta \KBless \vartheta\concat[n] \land \lh{\sigma} < n}  \\
                                        \end{cases}
 \end{equation}

 We then use the recursion theorem to transform this definition into a construction of the computable functional \( \Gamma \) which computes \( \Gamma^{\vartheta}_{\theta}(\sigma) \) from \( \zeron{\vartheta}, \vartheta, \theta, \sigma \) just as we did above for \( \Upsilon \).   

 \subsection{Verification}

 We will frequently argue by induction over \( \KBless \)  in this subsection so instead of invoking \cref{lem:hyp-set-has-min} every time we need to justify an inductive argument we will mention it only when we are dealing with a non-\( \piin{1} \) property so can't extend the induction beyond the well-ordered initial segment of \( U^{+} \). We start by proving that our trees are defined and respect our copy-length function.  

\begin{lemma}\label{lem:tree-tower-defined}
For all \( \sigma \in U^{+} \), \( T_{\sigma} \) is a nice tree (uniformly) computable from \( \zeron{\sigma} \) and if \( \tau \KBless \sigma \) then \( T_{\tau}\restr{l^{\sigma}_\tau} = T_{\sigma}\restr{l^{\sigma}_\tau} \) unless \( \tau \supfunneq \sigma \) in which case \( T_{\tau}\restr{l^{\sigma}_\tau} = T^{*}_{\sigma}\restr{l^{\sigma}_\tau} \supset T_{\sigma}\restr{l^{\sigma}_\tau} \).    
\end{lemma}
Here \( T^{*}_{\sigma} \) denotes the tree produced from \( T_{\sigma} \) via \cref{lem:tree-uniform-limit,lem:lim-of-mins} so that \( T^{*}_{\sigma}\restr{l(\sigma\concat[n])} \) is uniformly computable in \( \zeron{\eta(\sigma\concat[n])} \) and \( T_{\sigma}\restr{l(\sigma)} = T^{*}_{\sigma}\restr{l(\sigma)} \) .  
\begin{proof}
    As discussed above, the fact that \( \Upsilon(\zeron{\sigma}) \) is always either \( \wstrs \) or the result of applying \hyperref[lem:tree-minus-one]{the pulldown lemma} guarantees that \( T_{\sigma} \) is always defined, computable from \( \zeron{\sigma} \)  and a nice tree whenever \( \sigma \in U^{+} \). This leaves only the claim about restrictions to prove.  

    We first argue that the lemma holds when \( \tau^{-} = \sigma \) or \( \tau^{-} = \sigma^{-} \) since in either case we have \( l^{\sigma}_\tau = l(\tau) \) and in both cases  \( T_{\tau} \)  copies \( T^{*}_{\tau^{-}} \) on strings of length \( l(\tau^{-}) \).  If \( \sigma = \tau^{-} \) then we immediately have \( T_{\tau}\restr{l(\tau)} = T^{*}_{\sigma}\restr{l(\tau)} \) and if  \( \tau^{-} = \sigma^{-} \) then \( l(\sigma) > l(\tau) \) so both \( T_{\sigma} \) and \( T_{\tau} \) agree with \( T^{*}_{\tau^{-}} \) on strings of length up to \( l(\tau) \).

    Now suppose that \( \tau \supfun \sigma \) and that we've inductively established that the claim holds for \( \tau^{-} \supfunneq \sigma  \).  By \cref{lem:l-works} and \cref{eq:l-of-sigma,eq:l-of-sigma-tau} we will have \( l(\tau^{-}), l(\tau), l^{\sigma}_{\tau^{-}} \geq l^{\sigma}_{\tau} \) and by the inductive hypothesis we will have \( T_{\tau^{-}}\restr{l^{\sigma}_{\tau^{-}}} = T^{*}_{\sigma}\restr{l^{\sigma}_{\tau^{-}}} \).  Since \( T^{*}_{\tau^{-}}\restr{l(\tau^{-})} = T_{\tau^{-}}\restr{l(\tau^{-})} \) and \( T^{*}_{\tau^{-}}\restr{l(\tau)} = T_{\tau}\restr{l(\tau)} \) we can infer that \( T_{\tau}\restr{l^{\sigma}_{\tau}} = T^{*}_{\sigma} \).

    This leaves the case where \( \tau \) and \( \sigma \) are incompatible.  Let \( \tau', \sigma' \) be such that \( \tau'^{-} = \tau \meet \sigma = \sigma'^{-} \).  This case follows from the fact that we already proved the result for \( \tau', \sigma' \) combined with the argument in the prior paragraph that lets us translate that agreement from \( T_{\tau'} \) to \( T_{\tau} \) and \( T_{\sigma'}  \) to \( T_{\sigma} \).                        
\end{proof}

This result ensures that our uses of \hyperref[lem:tree-minus-one]{the pulldown lemma} are well-behaved, i.e., \( T_{\vartheta^{+}} \) really does extend the initial segment we are copying up to \( T_{\vartheta} \) from \( T_{\vartheta^{=}} \).  As \( \vartheta \meet \vartheta^{+} = \vartheta^{-} \) the result implies 
\[ T_{\vartheta^{+}}\restr{l(\vartheta)} = T^{*}_{\vartheta^{=}}\restr{l(\vartheta)} =  T_{\vartheta}\restr{l(\vartheta)}  \] 
We now argue that \( \Gamma \) is defined and induces homeomorphisms on the well-founded initial segment.

\begin{lemma}\label{lem:tree-gamma-defined}
If \( \theta \KBleq  \vartheta  \) then  \( \Gamma^{\vartheta}_{\theta} \Tleq \zeron{\vartheta} \) is an expansionary function mapping \( T_{\vartheta} \) into \( T_{\theta} \) that is the identity on strings of length up to \( l^{\vartheta}_\theta \).  Moreover, if \( \xi \in [\theta, \vartheta] \) then  \( \Gamma^{\vartheta}_\theta = \Gamma^{\xi}_{\theta} \circ \Gamma^{\vartheta}_\xi  \).       
\end{lemma}
We warn the reader that there is little, if any, insight offered by this proof as it consists almost entirely of straightforward expansion of \eqref{eq:gamma-rec-def} in the expected inductive arguments.
\begin{proof}
    All the equality cases, e.g., \( \theta = \vartheta \) or \( \xi = \vartheta \), are trivial so we ignore them for the rest of the proof.  To see that  \( \Gamma^{\vartheta}_{\theta} \) is computable in \( \zeron{\vartheta} \)  \textemdash rather than merely r.e. in \( \zeron{\vartheta} \) as our definition via the recursion theorem guarantees \textemdash it is enough to observe that our discussion immediately above ensures that our applications of \hyperref[lem:tree-minus-one]{the pulldown lemma} are passed a well-behaved \( T' \) and therefore \( \Gamma^{\epsilon^{+}}_\epsilon \Tleq \zeron{\epsilon} \) with the other cases following straightforwardly from \eqref{eq:gamma-rec-def}.  

    We now verify that  \( \Gamma^{\vartheta}_{\theta} \) has domain \( T_{\vartheta} \) and range contained in \( T_{\theta} \).  In \eqref{eq:gamma-rec-def} we explicitly ensured that \( \dom \Gamma^{\vartheta}_{\theta} \subset T_{\vartheta} \) so we just need to show that if \( \sigma \in T_{\vartheta} \) then \( \Gamma^{\vartheta}_{\theta}(\sigma)\conv \in T_{\theta} \).       Inductively suppose that this holds for all  \( \vartheta' \KBless \vartheta \).  The claim is immediate if \( \vartheta = \theta^{+} \) and follows straightforwardly by \eqref{eq:gamma-rec-def} and the inductive hypothesis if \( \vartheta \) is a successor, i.e., \( \vartheta \in  U^{+}\setminus U \).  Thus, we may assume that \( \vartheta \in U \).

    As \( \vartheta \) is the limit of its immediate extensions, there is some least \( m \) satisfying  \( \theta \KBless \vartheta\concat[m] \land \lh{\sigma} < m \) and by \eqref{eq:gamma-rec-def} we have \( \Gamma^{\vartheta}_{\theta}(\sigma) = \Gamma^{\vartheta\concat[m]}_{\theta}(\sigma) \).  As \( T_{\vartheta\concat[m]}\restr{l(\vartheta\concat[m])} = T^{*}_{\vartheta}\restr{l(\vartheta\concat[m])} \) and \( l(\vartheta\concat[m]) \geq m > \lh{\sigma} \) we have \( \sigma \in T_{\vartheta\concat[m]} \) and by the inductive hypothesis \( \Gamma^{\vartheta\concat[m]}_{\theta}(\sigma)\conv \in T_{\theta} \).  Thus, \( \Gamma^{\vartheta}_{\theta}(\sigma)\conv \in T_{\theta} \) as desired.

    We verify all the remaining aspects of the lemma via a similar \textemdash but simultaneous \textemdash induction on \( \vartheta \).  We start with the claim that \( \Gamma^{\vartheta}_{\theta} \) is always expansionary \textemdash though we actually verify the stronger claim that if \( \sigma, \sigma\concat[x] \in T_{\vartheta} \) then  \( \Gamma^{\vartheta}_{\theta}(\sigma\concat[x]) \supfun \Gamma^{\vartheta}_{\theta}(\sigma)\concat[x] \). As this property is clearly preserved under composition and  \hyperref[lem:tree-minus-one]{the pulldown lemma} guarantees \( \Gamma^{\epsilon^{+}}_{\epsilon} \) always has this property it is enough to handle the case where \( \vartheta \in U \).

    By \eqref{eq:gamma-rec-def} there are \( m, m' \) such that \( \Gamma^{\vartheta}_{\theta}(\sigma) =  \Gamma^{\vartheta\concat[m]}_{\theta}(\sigma) \) and \( \Gamma^{\vartheta}_{\theta}(\sigma\concat[x]) =  \Gamma^{\vartheta\concat[m']}_{\theta}(\sigma\concat[x]) \).  We always have either \( m' = m \) or \( m' = m +1 \) depending on whether \( \theta \KBless \vartheta\concat[\lh{\sigma}] \) and in the former case the result follows immediately from the inductive hypothesis.  In the later case, we can apply our inductive hypothesises with respect to composition to see that
     \begin{equation}\label{eq:tree-gamma-m-to-mprime}
        \Gamma^{\vartheta}_{\theta}(\sigma\concat[x]) = \Gamma^{\vartheta\concat[m']}_{\theta}(\sigma\concat[x]) = \left(\Gamma^{\vartheta\concat[m]}_{\theta} \circ \Gamma^{\vartheta\concat[m']}_{\vartheta\concat[m]}\right)(\sigma\concat[x]) 
    \end{equation}    
    By the inductive hypothesis for this claim, the fact that \( l^{\vartheta\concat[m']}_{\vartheta\concat[m]} \geq m \) and the inductive hypothesis about restrictions we know that the right hand side of \eqref{eq:tree-gamma-m-to-mprime} extends
    \begin{equation}\label{eq:tree-gamma-m-to-mprime-second}
    \Gamma^{\vartheta\concat[m]}_{\theta}\left(\Gamma^{\vartheta\concat[m']}_{\vartheta\concat[m]}(\sigma)\right)\concat[x] =  \Gamma^{\vartheta\concat[m]}_{\theta}(\sigma)\concat[x] = \Gamma^{\vartheta}_{\theta}(\sigma)\concat[x]
    \end{equation}
    This gives us the desired result that \(  \Gamma^{\vartheta}_{\theta}(\sigma\concat[x]) \supfun \Gamma^{\vartheta}_{\theta}(\sigma)\concat[x]  \).   We now verify the fact that \( \Gamma^{\vartheta}_{\theta} \) is the identity on strings of length up to \( l^{\vartheta}_\theta \).  The successor case follows from \hyperref[lem:tree-minus-one]{the pulldown lemma} and the discussion preceding this lemma.  For the limit case, it is enough to apply \eqref{eq:gamma-rec-def} and  observe that by \cref{def:good-copylen} and \cref{lem:l-works} for all \( m \) large enough to ensure \( \theta \KBless \vartheta\concat[m] \)  we have \( l^{\vartheta\concat[m]}_{\theta} \geq l^{\vartheta}_{\theta} \).

    For the moreover claim, the successor case is immediate from \eqref{eq:gamma-rec-def}.   Thus, we can assume \( \vartheta \in U \) and by \eqref{eq:gamma-rec-def} and the inductive hypothesis if \( m \) is the least element of \( \omega \) such that \( \lh{\sigma} < m \) and \( \theta \KBless \vartheta\concat[m] \) then \( \Gamma^{\vartheta}_{\theta}(\sigma) = \Gamma^{\vartheta\concat[m]}_{\theta}(\sigma)  \).  If we have \( \xi \KBless \vartheta\concat[m] \) then \( m \) must also be the least element of \( \omega \) such that  \( \lh{\sigma} < m \) and \( \xi \KBless \vartheta\concat[m] \) so the inductive hypothesis implies that 
    \[
        \Gamma^{\vartheta}_{\theta}(\sigma) = \Gamma^{\vartheta\concat[m]}_{\theta}(\sigma) =  \left(\Gamma^{\xi}_{\theta}\circ \Gamma^{\vartheta\concat[m]}_{\xi}\right)(\sigma) = \left(\Gamma^{\xi}_{\theta}\circ \Gamma^{\vartheta}_{\xi}\right)(\sigma)
    \]    
    To handle the case, where \eqref{eq:gamma-rec-def} tells us that  \( \Gamma^{\vartheta}_{\xi}(\sigma)  =  \Gamma^{\vartheta\concat[m']}_{\xi}(\sigma) \) for some \( m' > m \) we use the same technique of arguing \( \Gamma^{\vartheta\concat[m']}_{\vartheta\concat[m]}(\sigma) = \sigma \)  as in \eqref{eq:tree-gamma-m-to-mprime-second} to bridge the difference between \(  \Gamma^{\vartheta\concat[m']}_{\xi}(\sigma) \) and \(  \Gamma^{\vartheta\concat[m]}_{\xi}(\sigma) \).   
\end{proof}

\begin{lemma}\label{lem:homeo-almost}
If \( \theta \KBleq  \vartheta  \) then  \( \Gamma^{\vartheta}_{\theta}  \) is a homeomorphism of \( [T_{\vartheta}] \) with a subset of \( [T_\theta] \).  If \( \vartheta \in U^{+}_{\leftof f_U} \) then that subset is all of \( [T_\theta] \).  
\end{lemma}
\begin{proof}
    The first claim follows immediately from \cref{lem:expan-homeo} and the fact that  \( \Gamma^{\vartheta}_{\theta}  \) is expansionary.  Thus, it is enough to prove that  \( \Gamma^{\vartheta}_{\theta} \)  maps \( [T_{\vartheta}] \)  onto  \( [T_{\theta}] \) whenever \( \vartheta, \theta \in U^{+}_{\leftof f_U} \).  Given \( \theta \in U^{+}_{\leftof f_U}  \) suppose  by way of contradiction, \( \vartheta \KBgeq \theta \) is the least element of \( U^{+}_{\leftof f_U}  \) where  \( \Gamma^{\vartheta}_{\theta}  \) fails to be a surjection of \( [T_{\vartheta}] \)  with  \( [T_{\theta}] \).  Note that this property isn't \( \piin{1} \) so we can't use \cref{lem:hyp-set-has-min} to extend this argument beyond the genuinely well-ordered initial segment  of \( U^{+} \).    

    Trivially, we must have \( \vartheta \neq \theta \) and it is easy to see from \eqref{eq:gamma-rec-def} that \( \vartheta \) can't be a successor since \hyperref[lem:tree-minus-one]{the pulldown lemma} always builds a homeomorphism of \( T_{\epsilon} \) and \( T_{\epsilon^{+}} \).  Thus, we may suppose that \( \vartheta \in U \). Given \( f \in [T_{\theta}] \) define \( \sigma_n \) for all \( n  \) large enough that \( \theta \KBless \vartheta\concat[n] \) to satisfy
    \[
    \lh{\sigma_n} = n - 1 \land \sigma_n \in T_{\vartheta\concat[n]} \land \Gamma^{\vartheta\concat[n]}_\theta(\sigma_n) \subfun f
    \]    
    By the inductive hypothesis applied to \( \vartheta\concat[n] \), \( f \) must be the image of some path in \( T_{\vartheta\concat[n]} \)  so \( \sigma_n \) must exist for all large enough \( n \).  By \eqref{eq:gamma-rec-def} we have that \( \Gamma^{\vartheta\concat[n]}_\theta(\sigma_n) = \Gamma^{\vartheta}_\theta(\sigma_n) \) and by  \cref{lem:tree-gamma-defined} we know that  \( \Gamma^{\vartheta}_\theta \) is expansionary.  Therefore, \( g = \Union \sigma_n \) is a well-defined path through \( T_{\vartheta} \) with \( \Gamma^{\vartheta}_\theta(g) = f \).  As \( f \) was an arbitrary path through \( T_{\theta} \) this contradicts our assumption.             
\end{proof}

It is worth taking a moment to discuss, what goes wrong when \( \vartheta \) isn't part of the well-founded initial segment of \( U^{+} \). We know it is at least possible to have a failure of surjectivity since \textemdash if we set \( T_{\estr} \) to be a tree with only countably many paths \textemdash its paths can't be surjectively mapped onto the perfect set of paths through the root of our tower (\( T_0 \) only has non-hyperarithmetic paths).  But if we build each \( T_{\theta} \) to have the same paths as \( T_{\theta^{+}} \) where do these extra paths come from?  

In some sense they `come' from infinity.  Consider some \( \sigma_0 \in T_{\theta_0} \) which isn't extended by any path in \( T_{\theta_0} \).  If we map \( \sigma_0 \) down to \( T_{\theta_1} \) with \( \theta_1 \KBless \theta_0 \) then its image \(\sigma_1 = \Gamma^{\theta_0}_{\theta_1} \) is likely to grow in length.  As long as we are in the well-ordered initial segment this isn't a problem since we can only grow \( \sigma_0 \) finitely many times but outside of \( U^{+}_{\leftof f_U}  \) we can have an infinite descending sequence which grows \( \sigma_0 \) into an infinite path through \( T_0 \).

\begin{lemma}\label{lem:is-subg-tower}
The trees \( T_{\vartheta} \) along with the functional \( \Gamma \) and the set  \( U^{+}_{\leftof f_U}  \) form a subgeneric tower of height \( \wck \).  
\end{lemma}
Formally speaking, we should take \( I \) to be the  image of \( U^{+}_{\leftof f_U}  \) under \( \rho \) and likewise replace strings from \( U^{+}_{\leftof f_U}  \) with their images under \( \rho \).  However, we continue to elide the distinction between strings in \( U^{+} \) and their images under \( \rho \).  
\begin{proof}
We've already established every aspect of \hyperref[def:tower]{the definition of a tower} in \cref{lem:tree-tower-defined,lem:tree-gamma-defined,lem:homeo-almost} and we know that \( U^{+}_{\leftof f_U}  \) is (equivalent to) a path through \( \kleeneO \).  This leaves only establishing that we satisfy \cref{def:subgeneric-tower}, i.e., that our tower is a subgeneric tower.


Part \ref{def:subgeneric-tower:incompat} of \cref{def:subgeneric-tower} \textemdash the fact that distinct paths through  \( [T_{\vartheta}] \) are Turing incompatible over \( \zeron{\vartheta} \)  \textemdash  is directly guaranteed by \hyperref[lem:tree-minus-one]{the pulldown lemma} as is part \ref{def:subgeneric-tower:jump} \textemdash the fact that paths through \( T_{\vartheta^{+}} \) are  the jumps of paths through \( T_{\vartheta} \) (adjoining \( \zeron{\vartheta^{+}} \) and \( \zeron{\vartheta} \) respectively).  Since all of our trees have multiple paths part \ref{def:subgeneric-tower:non-trivial} \textemdash the fact that paths through \( T_{\vartheta} \) aren't computable in \( \zeron{\vartheta} \)  \textemdash follows from part  \ref{def:subgeneric-tower:incompat}.   \hyperref[lem:tree-minus-one]{The pulldown lemma} also ensures that every path through \( T_{\vartheta} \) either meets or strongly avoids on \( T_{\vartheta} \) every set of strings r.e. in \( \zeron{\vartheta} \) vindicating part \ref{def:subgeneric-tower:genericity} of \cref{def:subgeneric-tower}.  This only leaves part \ref{def:subgeneric-tower:splitting-subtree} to demonstrate.

Suppose  \( \theta \KBless \vartheta \), \( \Omega \) is a \( \zeron{\theta} \) computable functional and \( T^{\xi}_{\theta} \) is \( \Omega \) splitting over \( f \in [T_{\theta}] \) for all \( \xi \in [\theta, \vartheta) \) \textemdash where \( T^{\xi}_{\theta} \) is the image of \( T_{\xi} \) under \( \Gamma^{\xi}_{\theta} \).  Part \ref{def:subgeneric-tower:splitting-subtree} of \cref{def:subgeneric-tower} requires we demonstrate that \( T^{\vartheta}_\theta \) is \( \Omega \) splitting over \( f \). 

If \( \vartheta \) is a successor then the claim follows immediately by part \ref{lem:tree-minus-one:splitting} of \hyperref[lem:tree-minus-one]{the pulldown lemma} since that ensures that if  \( T_{\epsilon} \) is \( \Omega \) splitting over \( g \in [T_{\xi}] \) and \( \vartheta=\xi^{+} \) then \( T^{\vartheta}_{\xi} \) is \( \Omega \) splitting over \( g \).  By taking \( g \) to satisfy \( \Gamma^{\xi}_{\theta}(g) = f \) and using the assumption that \( T^{\xi}_{\theta} \) is \( \Omega \) splitting over \( f \) gives us the desired result.  Thus, we can assume \( \vartheta \in U \) and that there is some \( n_0 \) with \( \theta \KBless \vartheta\concat[n_0] \).           

Fix some index \( i > l(\vartheta\concat[n_0]) \) such that \( \recfnl{i}{\zeron{\vartheta\concat[n]}}{\sigma} = \Omega(\Gamma^{\vartheta\concat[n]}_{\theta}(\sigma)) \) for all \( n \geq n_0 \) \textemdash recall that we can effectively recover \( \vartheta\concat[n] \) from \( \zeron{\vartheta\concat[n]} \).  As \( l(\vartheta\concat[n]) \equiv 0 \pmod{4} \)  we can choose \( n > n_0 \) such that \( l(\vartheta\concat[n]) < 4i+2 < l(\vartheta\concat[n+1]) \) and \( g \in [T_{\vartheta}] \) such that \( \Gamma^{\vartheta}_{\theta}(g) = f \).  We now want to argue that the \( i \)-splitting pair \hyperref[lem:tree-minus-one]{the pulldown lemma} tries to capture when building \( T_{\vartheta\concat[n]} \) is the image of elements in \( T_{\vartheta} \) so to that end let \( \xi  \) be the successor of \( \vartheta\concat[n] \).  As \( \xi \supfun \vartheta\concat[n+1] \) by \cref{lem:l-works} we have \( l^{\vartheta}_{\xi} \geq l(\vartheta\concat[n+1]) > 4i+2 \) so by \cref{lem:tree-gamma-defined} 
\[
 \tau \in \set{g\restr{4i+2}, g\restr{4i+2}\concat[0], g\restr{4i+2}\concat[1]} \implies \Gamma^{\vartheta}_{\theta}(\tau) = \Gamma^{\xi}_{\theta}\left(\Gamma^{\vartheta}_{\xi}(\tau)\right) = \Gamma^{\xi}_{\theta}(\tau)
\]
By the inductive hypothesis applies to \( T^{\vartheta\concat[n]}_{\theta} \) there must be \( \sigma_0, \sigma_1 \in T_{\vartheta\concat[n]} \) extending \( \Gamma^{\xi}_{\vartheta\concat[n]}(g\restr{4i+2}) \) \textemdash by the above this maps down to an initial segment of \( f \)  \textemdash whose images under \( \Omega \circ \Gamma^{\vartheta\concat[n]}_{\theta} = \recfnl{i}{\vartheta\concat[n]}{} \) disagree.  Since we build \( \Gamma^{\xi}_{\vartheta\concat[n]} \) via \hyperref[lem:tree-minus-one]{the pulldown lemma}  part \ref{lem:tree-minus-one:splitting} of that lemma ensures that the images of \( g\restr{4i+2}\concat[0], g\restr{4i+2}\concat[1] \) under  \( \Omega \circ \Gamma^{\vartheta\concat[n]}_{\theta} \) are incompatible.    Since we can choose our index \( i \) to be arbitrarily large \textemdash and therefore \( g\restr{4i+2} \) arbitrarily long \textemdash  this entails that \( T^{\vartheta}_{\theta} \) is \( \Omega \) splitting over \( f \).       
\end{proof}

We can finally complete our proof of \hyperref[thm:non-standard]{the non-standard version of the main theorem}.  By \cref{prop:subgeneric-tower-properties}, since \( T_0  \) \textemdash i.e.,  \(  T_{\eta(\estr)} \) \textemdash  is the root of a subgeneric tower of height \( \wck \), we know that \( [T_0] \) is uniformly subgeneric.  This leaves only proving that \( [T_0] \) contains a subset homeomorphic to \( \baire \) via an expansionary function.  Since \( T_{\estr} = \wstrs \) if is enough to observe that by  \cref{lem:tree-gamma-defined}  \( \Gamma^{\estr}_{\eta(\estr)} \) is an expansionary homeomorphism of \( [T_{\estr}] \) with a subset of \( [T_0] \).  

We do note one complication here.  Recall from \cref{ssec:defining-zero-theta} that since \( U^{+} \) isn't actually well-founded under \( \KBless \) there won't be a unique way to define \( \zeron{\estr} \) and thus there won't be a unique definition of  \( \Gamma^{\estr}_{\eta(\estr)} \).  However, all the proofs given work for any of the paths through \( U \), i.e., any of the ways we might assign the sets \( \zeron{\vartheta} \) for \( \vartheta \in U^{+} \).  To establish \cref{thm:non-standard} it is enough to use any version of \( \Gamma^{\estr}_{\eta(\estr)} \)  defined by a path through \( U \) but it does raise the following question.

\begin{question}
Does the image of \( T_{\estr} \) under \( \Gamma^{\estr}_{\eta(\estr)} \) depend on the choice of definition of \( \zeron{\estr} \)?  
\end{question} 

As we only sketched a proof of \cref{thm:standard} in \cref{ssec:build-std-tower} we now show the construction used to prove \cref{thm:non-standard} can be modified to prove that result as well.

\subsection{Standard Verification}

The only real obstacle we face here is showing that we can translate ordinal notations into computable trees (of the appropriate form).  Even though we know that every ordinal below \( \wck \) can be identified with an element of the tree \( U^{+} \) we built above we can't guarantee that the restriction of \( U^{+} \) to the elements less than or equal to some \( \vartheta  \in U^{+} \) will also be a tree.

\begin{lemma}\label{lem:notation-to-tree}
For all limit notations \( \alpha \) there is a computable tree \( U \) and a computable function \( \rho \) which bijectively maps \( U^{+} \) under \( \KBless \)  to \( \set{\beta}{\beta \Oleq \alpha} \) in an order-preserving way.  This result holds with total uniformity.  
\end{lemma} 
We don't extend this result to other notations, even though it would be straightforward, because that would require giving up the assumption \textemdash used throughout the construction \textemdash that the `notations' we are working with take the of \( U^{+} \) for some computable tree \( U \). 
\begin{proof}
We define \( U \) and \( \rho \) as the limits of their stagewise approximations. We start at stage \( 0 \) by placing  \( \estr \in U_0  \) and defining \( \rho_0(\estr) = \alpha \).  At every stage \( s > 0    \) we permanently exclude any string with code less than \( s \) not in \( U_{s -1} \)  from ever entering \( U \) and then attend to every \( \sigma \in U_{s -1} \) as follows.  

Let \( \beta = \rho_{s-1}(\sigma) \) and let \( \gamma  \) be \( \rho_{s-1}(\sigma\concat[x]) \) for the largest \( x \) such that \( \sigma\concat[x] \in U_{s -1} \) if any such \( x \) exist and \( 0 \) otherwise.  Define \( \beta'_n \) to be the unique non-successor notation such that \( \Olim{\beta}{n} \) is a finite successor of \( \beta'_n \) \textemdash in other words \( \beta'_n \) is the result of recursively subtracting \( 1 \) from \( \Olim{\beta}{n} \) until we no longer have a successor notation.  

We check if there is some \( m < s \) such that  \( \gamma \Oless \beta'_m \) \textemdash this is computable as all notations below \( \alpha \) must be compatible.  If there is such an \( m \), enumerate \( \sigma\concat[l] \) for some large \( l \) into \( U_s \) and define \( \rho_s( \sigma\concat[l]) = \beta'_m  \) where \( m \) is the least such value.  This completes the task of attending to \( \sigma \) and completes our construction.

We've clearly defined a computable tree \( U \) and computable function \( \rho \) bijecting  \( U \) with \( \set{\beta}{\beta \Oleq \alpha \land \beta \in \kleeneO-}  \), the set of limit notations below or equal to \( \alpha \).  It is easy to check that if we order \( U \) using \( \KBless \) then  \( \rho \) is an order-preserving map and the same argument we used in \cref{ssec:prob71} allows us to extend \( \rho \) to \( U^{+} \) and turn it into an order isomorphism of \( U^{+} \) under \( \KBless \) and \( \set{\beta}{\beta \Oleq \alpha} \).  This completes the proof as the uniformity is evident from the construction.               
\end{proof}   

To prove \cref{thm:standard} we need to show that (with complete uniformity), given \( \alpha \in \kleeneO \) and \( T \Tleq \zeron{\alpha} \)  we can build a computable tree \( T_0 \) such that \( [T_0] \) is uniformly \( \alpha \)-subgeneric and homeomorphic via a \( \zeron{\alpha} \) computable expansionary function to \( [T] \).  In the construction in \cref{ssec:nonstandard-construction}  we were guaranteed that \( T_{\estr} \) would be a nice tree. As we can only apply \hyperref[lem:tree-minus-one]{the pulldown lemma} to nice trees our first task therefore is to transform the tree \( T \) into a nice tree \( \hat{T} \).   

We define \( \zeta(\sigma) \) by  \( \zeta(\estr) = \estr \) and \( \zeta(\sigma\concat[x])=\zeta(\sigma)\concat[0]\concat[x] \) and also define the tree \( \hat{T} = \set{\zeta(\sigma), \zeta(\sigma)\concat[1]}{\sigma \in T} \).  Clearly \( \hat{T} \) is a \( \zeron{\alpha} \) computable tree with \( [\hat{T}] \) homeomorphic to \( [T] \) via the computable expansionary function \( \zeta \).  We'd like to use \cref{lem:notation-to-tree} to define \( U^{+} \) and  set \( T_{\estr} = \hat{T} \) but we can only do this if \( \alpha \) is a limit notation.  If \( \alpha = \lambda \Oadd k \) for some non-zero \( k \) we simply apply \hyperref[lem:tree-minus-one]{the pulldown lemma} \( k \) times (with \( l =0 \)) starting with \( \hat{T} \) to generate the top \( k \) levels in our subgeneric tower and the appropriate expansionary homeomorphism.  If \( \lambda = 0\) then we are done.   

Thus, we can suppose that \( \alpha \) is a limit notation and use \cref{lem:notation-to-tree} to convert \( \alpha \) into \( U^{+} \).  We've already made sure that \( \hat{T} \) is nice and now we define \( T_{\estr} \) from \( \hat{T} \) using \cref{lem:tree-uniform-limit} to ensure that \( T_{\estr} \restr{l(\str{x})} \) is uniformly \( \zeron{\eta(\str{x})} \) computable whenever \( \str{x} \in U^{+} \).  This is the same transformation we used for every level in our tower but the top in the prior construction.  We can now just run that construction (and borrow its verification) to build a subgeneric tower of height \( \alpha \).    

As \( U^{+} \) is actually well-ordered by \( \KBless \) we don't need to worry about multiple definitions of \( \zeron{\sigma} \) and  \( \Gamma^{\estr}_{\eta(\estr)} \)  will be a homeomorphism of \( [T_{\estr}] \) and \( [T_{\eta(\estr)}] \).  Composing this with any expansionary homeomorphisms we build before defining \( T_{\estr} \), e.g. \( \zeta \) or the result of a finite number of applications of the pulldown lemma, verifies the existence of the desired \( \zeron{\alpha} \) computable homeomorphism of \( [T] \) and \( [T_0] \).  This suffices to prove the desired result as the uniformity of this construction is evident.

\appendix
\section{Definition of Ordinal Notations}\label{app:defo}

Taking for granted some effective system for coding successor and limit notations we define \( \kleeneO* \) and \( \kleeneO \) as follows \textemdash where \textbf{only} in this equation we understand membership in \( \kleeneO+ \) and \( \kleeneO- \) to be the computable sets of successor and limit codes respectively.

\begin{align}\label{eq:def-o}
\begin{split}
\alpha \in \kleeneO* \iff (\forall \beta, \beta' & \in \set{\beta \Oleq \alpha} )\Bigl[\beta = 0 \lor \beta \in \kleeneO+ \lor \beta \in \kleeneO- \\
                   & \land  \beta \neq \beta' \implies \beta \Oless \beta' \lor \beta' \Oless \beta \\
                   & \land \beta \in \kleeneO+ \implies \exists(\delta \Oless \beta)\left(\beta = \delta \Oadd 1\right) \\
                   & \land  \beta \in \kleeneO- \land \beta' \Oless \beta \iff \exists(n \in \omega)\left(\beta' \Oleq \Olim{\beta}{n}  \right) \\
                   & \land \beta \in \kleeneO- \implies \forall(n \in \omega)\left(\Olim{\beta}{n}\conv \right) \Bigr] \\        
\end{split} \\
\begin{split}
\alpha \in \kleeneO \iff  \alpha \in \kleeneO* & \land \lnot \exists(f \in \baire)\forall(n \in \omega) \Bigl[ \\
&  f(n) \in \set{\beta \Oleq \alpha} \land f(n) \Oless f(n+1) \Bigr] \\
\end{split}
\end{align}

\printbibliography

\end{document}